\documentclass[11pt]{article}

\usepackage{paralist}
\usepackage[margin=3.0cm]{geometry} 
\usepackage{graphicx} 

\usepackage{amssymb, amsfonts, amsmath, amscd, amsthm} 
\usepackage{mathtools} 
\usepackage{titlesec} 
\usepackage{csquotes} 
\usepackage[shortlabels]{enumitem} 
\usepackage{tikz-cd} 
\usepackage{wrapfig} 
\usepackage{epigraph} 
\usepackage{hyperref} 
\usepackage{todonotes} 
\usepackage{comment}
\usepackage{mathrsfs} 
\usepackage{todonotes}
\usepackage{soul}
\usepackage[misc]{ifsym}
\usepackage{multicol}

\hypersetup{
     colorlinks   = true,
     citecolor    = green
}

\numberwithin{equation}{section}

\titleformat*{\section}{\Large \scshape\center}
\titleformat*{\subsection}{\fontsize{14}{14} \sffamily}

\theoremstyle{plain}
\newtheorem{theorem}{Theorem}[section]
\newtheorem*{theorem*}{Theorem}

\newtheorem{lemma}[theorem]{Lemma}
\newtheorem{proposition}[theorem]{Proposition}
\newtheorem{corollary}[theorem]{Corollary}

\theoremstyle{definition}

\newtheorem*{definition*}{Definition}
\newtheorem{example}[theorem]{Example}

\theoremstyle{remark}
\newtheorem{remark}{Remark}
\newtheorem{question}{Question}

\DeclareMathOperator{\tr}{tr}
\DeclareMathOperator{\re}{Re}
\DeclareMathOperator{\im}{Im}

\begin{document}
\pagenumbering{gobble}
\title{Operators in the Fock-Toeplitz algebra}
\author{Wolfram Bauer\footnote{The work of the first author was partially supported by the organizers of ICSAT - 2023, CUSAT.}, Robert Fulsche and Miguel Angel Rodriguez Rodriguez}
\date{\emph{Dedicated to the memory of Nikolai Vasilevski}}
\maketitle
\pagenumbering{arabic}
\begin{abstract}
    We consider various classes of bounded operators on the Fock space $F^2$ of Gaussian square integrable entire functions over the complex plane. 
These include Toeplitz (type) operators, weighted composition operators, singular integral operators, Volterra-type operators and Hausdorff operators and range from classical 
objects in harmonic analysis to more recently introduced classes. As a leading problem and closely linked to well-known compactness characterizations we 
pursue the question of when these operators are contained in the Toeplitz algebra. This paper combines 
a (certainly in-complete) survey of the classical and more recent literature including new ideas for proofs from the perspective of quantum harmonic analysis (QHA). Moreover, 
we have added a number of new theorems and links between known results.

\noindent \textbf{Keywords:} Toeplitz operators, weighted composition operators, singular integral operators, generalized Volterra-type operators

\noindent \textbf{AMS subject classification:} 30H20, 46E22, 47B35, 47B33
\end{abstract}

\section{Introduction}
Let $F^2$ be the {\it Fock space} or {\it Segal-Bargmann space} of all Gaussian square integrable entire functions over the complex plane. Various classes of 
bounded and unbounded operators, as well as, Banach or $C^*$-algebras generated by them have been studied in the last decades. Partly such operators arise 
in models of mathematical physics such as Berezin-Toeplitz quantization or are of interest from the view point of (quantum) harmonic analysis, pseudo-differential 
operators or spectral theory of i.g.\ non-selfadjoint operators. Typical questions in the investigation of the individual operators concern boundedness or compactness characterizations or the inclusion in certain operator ideals (e.g.\ symmetric normed ideals such as the Schatten-$p$-classes), see for example \cite{Bauer_Fulsche2020, Bauer_Isralowitz2012,Berger_Coburn1994,Cao_Li2020,Fulsche2020,Galanopoulos_Stylogiannis2023,Grudsky_Vasilevski2002,Isralowitz_Virtanen_Wolf2015,Le2014}. While such questions are classical and have been studied for many years in the framework of other Hilbert spaces with a reproducing kernel (e.g.\ Bergman or Hardy spaces) a number of specific methods are available on the Fock space that can be used in the investigation and make it possible to answer some questions in the first place. To be more precise, Toeplitz operators on $F^2$ are unitarily equivalent via the Bargmann transform to pseudodifferential operators in Weyl quantization, a class for which an elaborated theory is available. Based on the approach in \cite{werner84} the so-called {\it quantum harmonic analysis} (QHA) and the {\it correspondence principle} have recently been established as powerful tools in the operator theory in (poly)-Fock spaces, see \cite{Fulsche2020,Fulsche2024,Fulsche_Hagger2023}.  
\vspace{1mm}\par 
The {\it Fock-Toeplitz algebra} (shortly: Toeplitz algebra) $\mathcal{T}$ over $F^2$ is the $C^*$-algebra generated by all Toeplitz operators with essentially bounded symbols. 
This $C^*$-algebra is invariant under operator shifts and in particular appears naturally in QHA. As is well-known,  the Toeplitz algebra plays an important role in the problem of 
compactness characterizations via the Berezin transform \cite{Bauer_Isralowitz2012}, and is linked to the notion of {\it (sufficiently) localized operators} \cite{Xia2015}. 
Various characterizations of $\mathcal{T}$ that partly do not use the notion of Toeplitz operators and are often easier to handle in applications are known in the literature \cite{Bauer_Fulsche2020,Fulsche2020,Xia2015}. In particular, it has been shown that the Toeplitz algebra $\mathcal{T}$ is strictly located between the compact and bounded operators 
$\mathcal{K}(F^2) \subsetneq \mathcal{T} \subsetneq \mathcal{L}(F^2)$.
\vspace{1mm}\par 
 In this paper we consider various classes of operators defined on the Fock space and previously known in the literature. Besides Toeplitz(-type) operators themselves these classes include {\it weighted composition operators, singular integral operators, generalized Volterra-type operators or Hausdorff operators}. The following two questions naturally arise: 
\begin{itemize}
\item[(A)] When is a given operator in one of the above classes contained in $\mathcal{T}$, and therefore can be approximated by localized operators?
\item[(B)] Under which conditions is an operator in one of the above classes even a Toeplitz operator with a bounded symbol and (in this case) how to compute this 
symbol from the mapping properties of the operator?
\end{itemize}
\par 
Finding an answer to (B) typically is more challenging and obviously implies a result in relation to question (A). To give an answer to question (A), we may use a characterization  of the Toeplitz algebra $\mathcal{T}$ that is particularly suited for the respective class of operators. 
\vspace{1mm}\par 
The present work combines a survey with a series of new results. In parts, known results are considered from the perspective of QHA, which leads to new proofs and reveals 
so far unexplored connections between different classes of operators. The contents and main results of the individual chapters are briefly described below. 
\vspace{1ex}\par 
In Section 2 we set up the notation. We give the definition and basic properties of the Berezin transform and bivariate Berezin transform of operators on $F^2$, as well as their canonical integral representations. 

The main results of the paper are contained in Section 3. The discussion starts with a reminder of the notion of Toeplitz operators and the Toeplitz algebra $\mathcal{T}$ along with some of its important properties. Specific examples of Toeplitz operators are the unitary Weyl operators in Section 3.2, which play a crucial role in the definition of operator shifts and 
a characterization of $\mathcal{T}$ in {\sc Theorem \ref{thm:Toeplitz_eq_c1}}. 

Section 3.3 introduces the class of weighted composition operators $W_{\psi, \varphi}$ (see \cite{Le2014,Ueki2007}). First, we remind the reader of well-known boundedness 
and compactness characterizations. As a new result we present a lower bound on the distance of $W_{\psi, \varphi}$ to the Toeplitz algebra. It then is a corollary that 
the only non-compact bounded operators $W_{\psi, \varphi}$ in $\mathcal{T}$ are non-trivial constant multiples of Weyl operators. In addition we can characterize the membership of differences 
of two weighted composition operators $W_{\psi_1, \varphi_1}-W_{\psi_2, \varphi_2}$ in $\mathcal{T}$. {\sc Theorem \ref{thm:compequalstoeplitz}} provides an accessible sufficient condition for $W_{\psi, \varphi}$ being a Toeplitz operator $T_{f_{\psi, \varphi}}$ with an explicitly given bounded symbol $f_{\psi, \varphi}\in L^{\infty}(\mathbb{C})$. 
Hence we have found a quite satisfactory answer to (A) and (B) above for weighted composition operators. 

Singular integral operators $S_{\varphi}$ on the Fock space, which are studied in Section 3.4, were introduced in \cite{Zhu2015} and more recently have been further investigated in \cite{Cao_Li2020}. 
{\sc Theorem \ref{thm_singular-integral-operators-questions-A-B}} of the paper presents a complete answer to the questions (A) and (B) in case of operators $S_{\varphi}$. 
Afterwards, we relate the theory of singular integral operators to the theory of commutative Toeplitz $C^\ast$-algebras, showing in particular that the set of bounded singular integral operators coincides with the commutative von Neumann algebra of vertical operators. This, in turn, leads to some open questions for future work.

Section 3.5 deals with generalized Volterra-type operators $V_{(g, \varphi)}$ on $F^2$, \cite{Fekadiea2023,Mengestie_Worku2018}. Boundedness and compactness criteria are known and will be recalled at the beginning of this section. We answer question (A) in the case of an operator $V_{(g, \varphi)}$ in {\sc Theorem \ref{thm:volterra_toeplitzalg}}. As a corollary it is observed that non-compact operators $V_{(g, \varphi)}$ in $\mathcal{T}$ are necessary Fredholm operators of index 
${\rm ind} V_{(g, \varphi)}=-2$. 

Section 3.6 studies Toeplitz-type operators (a.k.a.\ {\it generalized Toeplitz operators} in \cite{Xu_Yu2023,Xu_Yu2024}). Methods from QHA are used to provide a compactness characterization of such operators via the vanishing of its Berezin transform at infinity. Moreover, it is shown that Toeplitz-type operators are unitarily equivalent to ordinary 
Toeplitz operators (with the same symbol) acting on some polyanalytic Fock space.  To the best knowledge of the authors this result is new. 

The paper concludes with an investigation of Hausdorff operators on $F^2$ in Section 3.7, see \cite{Galanopoulos_Stylogiannis2023}. Boundedness and compactness 
of these operators have been characterized in the literature and as a new result we show that a bounded Hausdorff operator automatically defines an element in the 
Toeplitz algebra $\mathcal{T}$ ({\sc Proposition \ref{prop:bounded_HO_T_algebra}}). Studying problem (B) in this setting leads to a Stieltjes moment problem which is left open 
for future investigations. 

In the paper we have discussed a certainly incomplete list of operators acting on the Fock space. There are various other operator classes on $F^2$ 
which have been considered in the literature but are not included here, e.g.\ see \cite{Dong_Zhu2022, Jin_Tang_Feng2022}. 

\section{Preliminaries}
We denote by $\mu$ the Gaussian measure
\begin{align*}
    d\mu(z) = \frac{1}{\pi}e^{-|z|^2}dz
\end{align*}
on $\mathbb C$, where $dz$ is the Lebesgue measure. Then, we set
\begin{align*}
    F^2 = L^2(\mathbb C, ~d\mu) \cap \operatorname{Hol}(\mathbb C),
\end{align*}
where $\operatorname{Hol}(\mathbb C)$ denotes the space of all entire functions. $F^2$ is called the \emph{Fock space} or {\it Segal-Bargmann space}, basic references on its properties are \cite{Zhu2012, Berger_Coburn1987, Janson_Peetre_Rochberg1987, Berezin1972, Berezin1974}. Of course, $F^2$ is a closed subspace of $L^2(\mathbb C, \mu)$, hence itself a Hilbert space with the inner product
\begin{align*}
    \langle f, g\rangle = \int_{\mathbb C}f(w) \overline{g(w)}~d\mu(w).
\end{align*}
The standard orthonormal basis of $F^2$ is given by the monomials
\begin{align*}
    e_n(z) = \frac{1}{\sqrt{n!}}z^n, \quad n \in \mathbb N_0.
\end{align*}
As is well-known, $F^2$ is a reproducing kernel Hilbert space over $\mathbb C$, and the reproducing kernel function is given by
\begin{align*}
    K_z(w) = K(w, z) = e^{w \cdot \overline{z}}.
\end{align*}
This means that $ f(z) = \langle f, K_z\rangle$ for every $z \in \mathbb C$ and $f \in F^2$. The norm of the reproducing kernel function can be explicitly computed and is given by 
$\| K_z\| = e^{\frac{|z|^2}{2}}$. Thus, the normalized reproducing kernel is
\begin{align*}
    k_z(w) = \frac{K_z(w)}{\| K_z\|} = e^{w \cdot \overline{z} - \frac{| z|^2}{2}}.
\end{align*}
A basic object in the operator theory on $F^2$ is the \emph{Berezin transform}, which is defined as follows: Given any operator $A$ on $F^2$ with domain $D(A)$
such that $k_z \in D(A)$ for every $z \in \mathbb C$, we define its  Berezin transform $\widetilde{A}$ as a function $\widetilde{A}: \mathbb C \to \mathbb C$ by
\begin{align*}
    \widetilde{A}(z) := \langle Ak_z, k_z\rangle.
\end{align*}
The Berezin transform is injective on a suitable class of linear operators which strictly contains the bounded linear operators. In the following, we will only deal with this strict subclass, which ensures that the Berezin transform always uniquely determines the operator. For such operators, we always have 
$\| \widetilde{A}\|_\infty \leq \| A\|_{\textup{op}}$. We note that 
$\widetilde{A}$ is a uniformly continuous (even real analytic) function.

There are other transforms on the space of bounded linear operators on $F^2$: For example, there is the \emph{bivariate Berezin transform}: For $w, z \in \mathbb C$ we set
\begin{align*}
    \widetilde{A}(w, z) := \langle Ak_w, k_z\rangle.
\end{align*}
Clearly, we have $\widetilde{A}(z) = \widetilde{A}(z,z)$, i.e., also the bivariate Berezin transform uniquely determines the operator (the injectivity of the Berezin transform is actually a consequence of the injectivity of the bivariate Berezin transform).

We also mention the {\it (canonical) integral kernel} of a bounded linear operator. For any $A \in \mathcal L(F^2)$ and $f \in F^2$, $z \in \mathbb C$ 
we can express $Af$ as an integral:
\begin{align*}
    Af(z) &= \langle Af, K_z\rangle = \langle f, A^\ast K_z\rangle\\
    &= \int_{\mathbb C} f(w) \overline{A^\ast K_z(w)} ~d\mu(w)\\
    &= \int_{\mathbb C} f(w) \overline{\langle A^\ast K_z, K_w\rangle}~d\mu(w)\\
    &= \int_{\mathbb C} f(w) \langle AK_w, K_z \rangle ~d\mu(w).
\end{align*}
Hence, setting $K_A(w,z) = \langle AK_w, K_z\rangle$, every bounded linear operator on $F^2$ is an integral operator:
\begin{align*}
    Af(z) = \int_{\mathbb C} K_A(w,z) f(w)~d\mu(w).
\end{align*}
Clearly, since 
\begin{align*}
    K_A(w,z) = e^{\frac{|z|^2 + |w|^2}{2}} \widetilde{A}(w,z),
\end{align*}
the canonical integral kernel is also uniquely determined by the operator. It is important to stress that, nevertheless, a linear operator can still be given by 
more than one different integral expression; we will see such examples below. In such cases, at most one of the expressions can be given by the canonical integral kernel. Note that we have well-known formulas for the integral kernel of the composition of two integral operators:
\begin{align*}
    K_{AB}(w,z) = \int_{\mathbb C} K_A(w,\xi)K_B(\xi,z)~d\mu(\xi).
\end{align*}
In particular, we have for the (bivariate) Berezin transforms:
\begin{align*}
    \widetilde{AB}(w,z) &= e^{-\frac{|z|^2+|w|^2}{2}} K_{AB}(w,z)\\
    &= e^{-\frac{|z|^2+|w|^2}{2}} \int_{\mathbb C} K_A(w,\xi) K_B(\xi, z)~d\mu(\xi)
    = \frac{1}{\pi} \int_{\mathbb C} \widetilde{A}(w,\xi) \widetilde{B}(\xi,z)~d\xi
\end{align*}
and
\begin{align*}
    \widetilde{AB}(z) = \frac{1}{\pi} \int_{\mathbb C} \widetilde{A}(z,\xi)\widetilde{B}(\xi, z)~d\xi.
\end{align*}
\section{Classes of operators on the Fock space}

In this section, we will study several specific classes of linear operators acting on the Fock space. More specifically, we investigate weighted composition operators (Section \ref{sec:weightedcomp}), singular integral operators (Section \ref{sec:singularint}), generalized Volterra-type operators (Section \ref{sec:volterra}), Toeplitz-type operators (Section \ref{sec:toeplitztype}) and Hausdorff operators (Section \ref{sec:Hausdorff}). For each of these classes of linear operators, results characterizing boundedness and compactness of the operators are available in the literature. We will provide brief reviews on these results in each section. Beyond reproducing known results, we will deal with the following questions for each of those classes: \emph{When are such operators contained in the Toeplitz algebra? When can such operators be written as Toeplitz operators?} For each of the above-mentioned classes, we obtain a full characterization regarding the first question. For the second question, we obtain some sufficient criteria. Before starting these investigations, we will give a short review on Toeplitz and Weyl operators on the Fock space in Sections \ref{sec:Toeplitz} and \ref{sec:Weyl}, respectively.

\subsection{Toeplitz operators}\label{sec:Toeplitz}

Toeplitz operators on the Fock space (a.k.a.\ {\it Berezin-Toeplitz operators}) have been studied intensively in the literature during the last decades (see, e.g., \cite{Cho_Park_Zhu2014, Bommier-Hato_Youssfi_Zhu2015, Hu_Lv2014, Hu_Lv2017, Isralowitz_Virtanen_Wolf2015, Esmeral_Maximenko2016, Esmeral_Vasilevski2016} and references therein) and in particular they 
arise in classical models of quantization. This class of operators is not closed under taking products, however, it defines the Toeplitz algebra $\mathcal{T}$ (see \cite{Bauer_Fulsche2020,Fulsche2020,Hagger2021,Xia2015}) which is a central object in this paper. We first recall the definition and state some fundamental result 
which provides a compactness characterization for bounded operators on $F^2$. 

\vspace{1mm}\par 
We denote by $P$ the orthogonal projection from $L^2(\mathbb C, \mu)$ onto $F^2$. In particular, we have for $z \in \mathbb C$ and $g \in L^2(\mathbb C, \mu)$:
\begin{align*}
    Pg(z) = \langle Pg, K_z\rangle = \langle g, K_z\rangle = \int_{\mathbb C} g(w) e^{z \cdot \overline w}~d\mu(w).
\end{align*}
For $f \in L^\infty(\mathbb C)$ we define the Toeplitz operator $T_f$ with symbol $f$ by
\begin{align*}
    T_f(g) = P(fg),
\end{align*}
i.e.
\begin{align*}
    T_f(g)(z) = \int_{\mathbb C} f(w) g(w) e^{z \cdot \overline w}~d\mu(w).
\end{align*}
Note that the (bivariate) Berezin transform is given by the following integral transforms (where $w, z \in\mathbb C$):
\begin{align}
    \widetilde{T_f}(w,z) &= \langle P( fk_w), k_z\rangle = \langle fk_w, k_z\rangle\notag \\
    &= e^{-\frac{|z|^2+|w|^2}{2}} \int_{\mathbb C} f(v) e^{v\cdot \overline w + z \cdot \overline v}~d\mu(v)\\
  \widetilde{f}(z):&=\widetilde{T_f}(z) = e^{-|z|^2} \int_{\mathbb C} f(v) e^{2\re(v \cdot \overline z)}~d\mu(v)
    = \int_{\mathbb C} f(v+z)~d\mu(v). 
\end{align}
As a consequence of the last identity we remark that the Berezin transform ``$\sim$'' on functions $f$ commutes with the {\it shifts} $\alpha_z(f)= f (\cdot -z)$. The canonical integral kernel is therefore
\begin{align}
    K_{T_f}(w,z) = \int_{\mathbb C} f(v) e^{v\cdot \overline w + z \cdot \overline v}~d\mu(v),
\end{align}
i.e., we have
\begin{align*}
    \int_{\mathbb C} g(w) \int_{\mathbb C} f(v) e^{v\cdot \overline w + z \cdot \overline v}~d\mu(v)~d\mu(w) = T_f(g)(z) = \int_{\mathbb C} f(v) g(v) e^{z \cdot \overline v}~d\mu(v).
\end{align*}
By $\mathcal T$ we denote the $C^\ast$-algebra generated by all Toeplitz operators with $L^\infty$ symbols and we refer to $\mathcal T$ as the \emph{Toeplitz algebra} (over $F^2$). 

Let $\overline{\mathcal{S}}$ denote the linear closure of a set of operators $\mathcal{S} \subset \mathcal{L}(F^2)$. 
An important fact regarding this algebra is the following:
\begin{theorem}[\cite{Xia2015, Fulsche2020}]
      $  \mathcal T = \overline{\{ T_f: ~f \in L^\infty(\mathbb C)\}}.$
\end{theorem}
As usual we denote by $\mathcal{K}(F^2)$ the ideal of compact operators on the Fock space. Let $C_0(\mathbb{C})$ be space of continuous complex-valued functions 
vanishing at infinity. 

The Toeplitz algebra shows up prominently in the following result concerning compactness:
\begin{theorem}[\cite{Bauer_Isralowitz2012}]\label{thm:compactness}
Let $A \in \mathcal L(F^2)$. Then, $A \in \mathcal K(F^2)$ if and only if $A \in \mathcal T$ and $\widetilde{A} \in C_0(\mathbb C)$.
\end{theorem}
There are several important facts to note from the previous theorem: Firstly, the Toeplitz algebra contains all compact operators. Secondly, the theorem yields a useful way to prove that an operator is not contained in the Toeplitz algebra: If it is not compact but the Berezin transform is in $C_0(\mathbb C)$, the operator cannot lie in $\mathcal T$. Finally, it is in practice not hard to check if the Berezin transform of an operator is vanishing at infinity. On the other hand, it seems a hard problem to check whether a bounded linear operator is contained in the Toeplitz algebra: How should one see whether an operator can be approximated by sums of products of Toeplitz operators? Addressing this question can be suitably done by using the class of operators described in the next section.  Before ending our discussion on Toeplitz operators, let us mention that it can be shown, by means of the \emph{Berger-Coburn estimates} \cite{Berger_Coburn1994}, that even bounded Toeplitz operators with unbounded symbols are often contained in $\mathcal T$. Without giving details on the definition of $T_f$ for unbounded $f$ (it is defined by the same integral expression on the space of all $g \in F^2$ for which the integral exists absolutely), we give the following result:
\begin{theorem}[{\cite[Theorem 4.19]{Fulsche2024}}]\label{thm:BergerCoburnToeplitzAlg}
    Let $f: \mathbb C \to \mathbb C$ be measurable such that $f K_z \in L^2(\mathbb C, \mu)$ for each $z \in \mathbb C$. Denote $\widetilde{f}^{(t)} = f \ast g_t$, where $g_t(z) = \frac{1}{\pi t}e^{-\frac{|z|^2}{t}}$.
If $\widetilde{f}^{(t)} \in L^\infty(\mathbb C)$ for some $0 < t < \frac{1}{2}$, then $T_f \in \mathcal T$.
\end{theorem}

\subsection{Weyl operators}\label{sec:Weyl}
For $z \in \mathbb C$ we define the Weyl operator $W_z$ by
\begin{align}\label{def:Weyloperator}
    W_z f(w) = k_z(w) f(w-z).
\end{align}
An easy substitution shows that the Weyl operators are isometric on $F^2$. Since $W_z f(w) \to f(w)$ pointwise as $z \to 0$, a well-known fact from measure theory ({\it Scheff\'{e}'s lemma}) shows that the assignment $z \mapsto W_z$ is continuous in strong operator topology. Moreover, $W_z$ is unitary and satisfies the identities
\begin{align*}
    W_z^\ast = W_z^{-1} = W_{-z}, \quad W_z W_w  = e^{-i\im(z \cdot \overline w)}W_{z+w}.
\end{align*}
Even though this is not at all clear when one encounters these operators, they are contained in the Toeplitz algebra. Even more, they are Toeplitz operators themselves: Upon letting
\begin{align}\label{Weyl-operator-equal-TO}
    g_z(w) = e^{2i\im(w\cdot \overline{z}) + \frac{1}{2}|z|^2}
\end{align}
it is $T_{g_z} = W_z$. 
\vspace{1mm}\par 
That the equality is true is, again, not at all obvious. But we have enough techniques available to show this: We will compare the Berezin transforms! The bivariate Berezin transform of a Weyl operator is given by
\begin{align}
    \widetilde{W_z}(w, v) &= \langle W_z k_w, k_v\rangle = \langle W_z W_w 1, k_v\rangle\notag \\
    &= e^{-i\im(z \cdot \overline w)} \langle W_{z+w} 1, k_v\rangle\notag \\
    &= e^{-i\im(z \cdot \overline w)} \langle k_{z+w}, k_v\rangle\notag \\
    &= e^{-i\im(z \cdot \overline w)} e^{-\frac{|z+w|^2 + |v|^2}{2}}\langle K_{z+w}, K_v\rangle\notag \\
    &= e^{-i\im(z \cdot \overline w)} e^{-\frac{|z+w|^2 + |v|^2}{2}} e^{v \cdot \overline{(z+w)}}\notag \\
    &= e^{-\frac{|z|^2 + |w|^2 + |v|^2}{2} - z\cdot \overline w + v \cdot \overline z + v\cdot \overline w}.
\end{align}
In particular,
\begin{align}
    \widetilde{W_z}(w) &= e^{-\frac{|z|^2}{2} + 2i\im (w \cdot \overline z)},\\
    K_{W_z}(w,v) &= e^{-\frac{|z|^2}{2} + v\cdot \overline{z} + v \cdot \overline{w} - z \cdot \overline{w}}
\end{align}
It is not hard to verify that $T_{g_z}$ has indeed the same Berezin transform. The Weyl operators are important for our purpose because they implement a particular group action of $\mathbb C$ on $\mathcal L(F^2)$: For $A \in \mathcal L(F^2)$ and $z \in \mathbb C$ set
\begin{align*}
    \alpha_z(A) = W_z A W_z^\ast .
\end{align*}
Note that for each $A \in \mathcal L(F^2)$, $z \mapsto \alpha_z(A)$ is continuous with respect to the strong operator topology. Here, the special role of the Toeplitz algebra comes into play:
\begin{theorem}[{\cite[Theorem 3.1]{Fulsche2020}}]\label{thm:Toeplitz_eq_c1}
It is
\begin{align*}
    \mathcal T = \{ A \in \mathcal L(F^2): ~z \mapsto \alpha_z(A) \text{ is } \| \cdot\|_{op}\text{-continuous}\}.
\end{align*}
\end{theorem}
We want to emphasize that there are several additional useful characterizations of the algebra $\mathcal T$, cf.\ \cite{Hagger2021, Bauer_Fulsche2020, Xia2015}.

Due to the previous result, we find it appropriate to quickly discuss how the ``shifts'' behave with respect to the Berezin transform and the canonical integral kernel. 
First, we observe: 
\begin{align*}
    \widetilde{\alpha_z(A)}(w,v) &= \langle W_z A W_{-z} k_w, k_v\rangle = \langle A W_{-z} W_w 1, W_{-z} W_v 1\rangle\\
    &= \langle e^{i\im(z\cdot \overline{w})} A W_{w-z} 1, e^{i\im(z \cdot \overline v)} W_{v-z} 1\rangle\\
    &= e^{i\im(z \cdot \overline{(w-v)})} \langle A k_{w-z}, k_{v-z}\rangle\\
    &= e^{i\im(z \cdot \overline{(w-v)})} \widetilde{A}(w-z, v-z).
\end{align*}
In particular, this gives
\begin{align*}
    \widetilde{\alpha_z(A)}(w) = \widetilde{A}(w-z)
\end{align*}
and
\begin{align*}
    K_{\alpha_z(A)}(w, v) &= e^{\frac{|w|^2+|v|^2}{2}} \widetilde{\alpha_z(A)}(w,v)\\
    &= e^{\frac{|w|^2+|v|^2}{2}} e^{i\im(z \cdot \overline{(w-v)})}\widetilde{A}(w-z, v-z)\\
    &= e^{\frac{|w|^2+|v|^2}{2}} e^{i\im(z \cdot \overline{(w-v)})} e^{-\frac{|w-z|^2 + |v-z|^2}{2}} K_A(w-z, v-z)\\
    &= e^{-|z|^2 + z\cdot \overline w + v \cdot \overline z} K_A(w-z, v-z).
\end{align*}
We want to give the following convenient sufficient criterion for proving that an operator is contained in $\mathcal T$ in terms of its bivariate Berezin transform:
\begin{proposition}[{\cite[Theorem 3.6]{Fulsche2023}}]\label{crit:Wiener_algebra}
    Assume $A \in \mathcal L(F^2)$ is such that there exists $H \in L^1(\mathbb C)$ satisfying
    \begin{align*}
        |\widetilde{A}(z, w)| \leq H(z-w).
    \end{align*}
    Then, $A \in \mathcal T$. 
\end{proposition}
\subsection{Weighted composition operators}\label{sec:weightedcomp}
Given a function $\psi \in F^2$ and a $\varphi \in \operatorname{Hol}(\mathbb C)$, we can define the \emph{weighted composition operator} $W_{\psi, \varphi}$ on $F^2$ by
\begin{align*}
    W_{\psi, \varphi}f = \psi \cdot f \circ \varphi, \quad f \in F^2.
\end{align*}
See \cite{Carroll_Gilmore2021, Le2014,Ueki2007} for historical notes. The functions $\varphi$ and $\psi$ are called the {\it symbol} and the 
{\it multiplier} of $W_{\psi, \varphi}$, respectively. Note that for $\psi = k_z$ and $\varphi(w) = w-z$, $W_{\psi, \varphi} = W_z$. We collect some important properties. The first characterization of boundedness, respectively compactness was obtained in \cite{Ueki2007} in terms of the quantity 
\begin{align*}
    B_\varphi(|\psi|^2)(z) :=& \| W_{\psi, \varphi} k_z\|^2\\
    =&\frac{1}{\pi} \int_{\mathbb C} |\psi(w)|^2 | e^{\varphi(w) \cdot \overline{z}}|^2e^{-|z|^2-|w|^2}~dw.
\end{align*}
A second characterization was obtained in \cite{Le2014} in terms of another quantity. Before introducing this quantity, we first want to note that 
\begin{align}\label{eq:w_adjoint}
W_{\psi, \varphi}^\ast K_z(w) &= \overline{\psi(z)}K_{\varphi(z)}(w).
\end{align}
Using this, it is not hard to verify that
\begin{align}\label{eq:mz}
    M_z(\psi, \varphi) :=& \| W_{\psi, \varphi}^\ast k_z\|^2
    = |\psi(z)|^2 e^{|\varphi(z)|^2 - |z|^2}.
\end{align}
In the folllowing we will also write 
$$M(\psi, \varphi) := \sup_{z \in \mathbb C} M_z(\psi, \varphi).$$

For the next two results we refer to \cite{Carroll_Gilmore2021, Le2014, Ueki2007} and references therein.
\begin{lemma}[{\cite[Proposition 2.1]{Le2014}}]\label{Lemma-wco-bounded}
When $0 \neq \psi \in F^2$ and $\varphi \in \operatorname{Hol}(\mathbb C)$ such that $M(\psi, \varphi) < \infty$, then there are $a, \lambda \in \mathbb C$, $|\lambda|\leq 1$, such that
\begin{align*}
    \varphi(w) = a + \lambda w. 
\end{align*}
When $|\lambda|= 1$, then also
\begin{align*}
    \psi(w) = \psi(0) e^{-\overline{a}\lambda w}
\end{align*}
and in this case, $M_z(\psi, \varphi)$ is constant and strictly positive (as a function of $z$).
\end{lemma}
Having defined all the necessary terms and objects now, here are the main results about the weighted composition operators: 
\begin{theorem}[\cite{Ueki2007, Le2014}]\label{Theorem-properties-weighted-composition-operators} 
Let $\psi, \varphi \in \operatorname{Hol}(\mathbb C)$. Then, the following are equivalent:
\begin{enumerate}
    \item[(1)] $W_{\psi, \varphi}$ is bounded on $F^2$;
    \item[(1')] $\psi \in F^2$ and $M(\psi, \varphi) < \infty$;
    \item[(1'')] $\sup_{z\in \mathbb C} B_\varphi(|\psi|^2)(z) < \infty$.
\end{enumerate}
    Assume that $W_{\psi, \varphi}$ is bounded such that $\varphi(w) = a+\lambda w$ with $|\lambda|\leq 1$. Then, the following are equivalent:
\begin{enumerate}
    \item[(2)] $W_{\psi, \varphi}$ is compact; 
    \item[(2')] $M_{z}(\psi, \varphi) \to 0$ as $|z| \to \infty$;
    \item[(2'')] $B_\varphi(|\psi|^2)(z) \to 0$ as $|z| \to \infty$.
\end{enumerate}
If $\lambda = 0$, then $W_{\psi, \varphi}$ is of rank one.
\end{theorem}
Note that (2') cannot be satisfied when $|\lambda| = 1$, hence $W_{\psi, \varphi}$ can only be compact when $|\lambda|< 1$.

We compute the Berezin transform, assuming that $W_{\psi, \varphi}$ is bounded: Then, $\varphi(w) = a + \lambda w$ and hence:
\begin{align}
    \widetilde{W_{\psi, \varphi}}(w,z) &= \langle W_{\psi, \varphi} k_w, k_z\rangle\notag \\
    &= e^{-\frac{|w|^2 + |z|^2}{2}}\langle \psi K_w(a + \lambda (\cdot)), K_z\rangle\notag \\
    &= e^{-\frac{|w|^2 + |z|^2}{2}} e^{a \cdot \overline{w}} \langle \psi K_{\overline{\lambda}w}, K_z\rangle\notag \\
    &= e^{-\frac{|w|^2 + |z|^2}{2}} e^{(a + \lambda z) \cdot \overline{w}} \psi(z) .
\end{align}
Hence, we get
\begin{align}
    \widetilde{W_{\psi, \varphi}}(z) = e^{(\lambda - 1)|z|^2 + a\cdot \overline z } \psi(z),\\
    K_{W_{\psi, \varphi}}(w,z) = e^{(a + \lambda z) \cdot \overline{w}} \cdot \psi(z).
\end{align}

We now discuss the following question: {\it When is $W_{\psi, \varphi} \in \mathcal T$?} Here are examples of operators $W_{\varphi, \psi}$ that are actually contained in $\mathcal T$:
\begin{proposition}\label{prop:compinT}
    Let $\psi, \varphi$ such that $W_{\psi, \varphi}$ is bounded on $F^2$. 
\begin{enumerate}
    \item If $|\lambda| = 1$, then $W_{\psi, \varphi} \in \mathcal T$ if and only if $\lambda = 1$.
    \item If $\lambda = 0$, then $W_{\psi, \varphi} \in \mathcal T$.
    \item If $0 < |\lambda| < 1$, then $W_{1, \varphi} \in \mathcal T$.
    \item If $0 < |\lambda| < 1$, then $M_{z}(\psi, \varphi) \to 0$ as $|z| \to \infty$ implies  $W_{\psi,\varphi} \in \mathcal T$.
\end{enumerate}
\end{proposition}
\begin{proof} 
\begin{enumerate}
    \item For $\lambda = 1$, then $W_{\psi, \varphi}= \psi(0) W_a$ is zero or a Weyl operator up to a factor, which is well-known to be in $\mathcal T$. If $|\lambda| = 1$ with $\lambda \neq 1$, then $W_{\psi, \varphi}$ is not compact (unless it is zero), but  $\widetilde{W_{\psi, \varphi}} \in C_0(\mathbb C)$, showing that the operator cannot be in $\mathcal T$.
    \item If $\lambda = 0$, then the operator is compact, hence in $\mathcal T$.
    \item If $0 < |\lambda| < 1$, then 
    \begin{align*}
        M_z(1, \varphi) = e^{(|\lambda|^2-1)|z|^2 + |a|^2 +2\operatorname{Re}(\lambda z \overline{a})} \in C_0(\mathbb C)
    \end{align*}
    as a function of $z$, hence the operator is compact. Therefore, $W_{1, \varphi} \in \mathcal T$. 
    \item Just as in the previous point, the assumption provides compactness of the operator, hence it is contained in $\mathcal T$. 
\end{enumerate}
\end{proof}

We observe that $W_{\psi, \varphi}$ with $|\lambda| = 1$ is never compact unless $W_{\psi, \varphi}=0$. For $\lambda = 0$, we have $W_{\psi, \varphi}f(w) = f(a) \psi(w)$, which is of rank one (or zero). In particular, $W_{\psi, \varphi} \in \mathcal K(F^2) \subset \mathcal T$. Hence, we in principle only need to consider $0 < |\lambda| < 1$. Nevertheless, we will now derive a more general result, which will show that the list in Proposition \ref{prop:compinT} is actually complete. 

Let $A \in \mathcal L(F^2)$ be \emph{sufficiently localized}, meaning that
\begin{align*}
    |\langle A k_z, k_w\rangle|\leq \frac{C}{(1+|z-w|)^\beta}
\end{align*}
for some $C > 0$ and $\beta > 2$ and every $z, w \in \mathbb C$. It is well-known that such operators are dense in $\mathcal T$ \cite{Xia2015}. Further, every Toeplitz operator is sufficiently localized. Clearly, $A^\ast$ is again sufficiently localized. We note that the following argument is similar to the statement and proof of \cite[Theorem 3]{Coburn2001}.

For such $A$ and $W_{\psi, \varphi}$ bounded with $\lambda \neq 1$ we estimate (using also Eq.\ \eqref{eq:mz}) for all $z \in \mathbb{C}$: 
\begin{align}\label{GL_estimate-difference_TO_and_weighted_composition_operator}
    \| W_{\psi, \varphi}\| \| A - W_{\psi, \varphi}\| &\geq \left |\langle A^\ast k_z, W_{\psi, \varphi}^\ast k_z\rangle - \langle W_{\psi, \varphi}^\ast k_z, W_{\psi, \varphi}^\ast k_z\rangle \right |\\
    &= \left | \langle A^\ast k_z, W_{\psi, \varphi}^\ast k_z \rangle - M_z(\psi, \varphi) \right |. \notag
\end{align}
We assume that $W_{\psi, \varphi}$ is bounded such that $\varphi(z) = a+\lambda z$ with $|\lambda| \leq 1, ~\lambda \neq 1$ (see Lemma \ref{Lemma-wco-bounded} and Theorem \ref{Theorem-properties-weighted-composition-operators}). Then, for $|z|$ sufficiently large, we always have $z \neq a + \lambda z$. Hence, using Eq.\ \eqref{eq:w_adjoint}, it follows with a suitable constant $C>0$: 
\begin{align}\label{Estimate-Toeplitz-weighted-comp}
    |\langle A^\ast k_z, W_{\psi, \varphi}^\ast k_z\rangle| &= |\psi(z)| e^{-\frac{|z|^2}{2}} |\langle A^\ast k_z, K_{\varphi(z)}\rangle| \notag \\
    &\leq |\psi(z)| e^{-\frac{|z|^2}{2} + \frac{|\varphi(z)|^2}{2}} \frac{C}{(1+|z-\varphi(z)|)^\beta} \notag \\
    &= |\psi(z)| e^{\frac{1}{2}(|\varphi(z)|^2 - |z|^2)} \frac{C}{(1 + |(1-\lambda)z - a|)^\beta}.
\end{align}
Now, the first factor of this product is bounded by Theorem \ref{Theorem-properties-weighted-composition-operators}, (1), as $W_{\psi, \varphi}$ is assumed bounded. The second factor clearly converges to zero as $z \to \infty$ (since we assumed $\lambda \neq 1$). Hence, we obtain that for each $A$ sufficiently localized and $W_{\psi, \varphi}$ bounded with $|\lambda| \leq 1, ~\lambda \neq 1$, the following holds true:
\begin{align*}
    |\langle A^\ast k_z, W_{\psi, \varphi}^\ast k_z\rangle| \to 0,  \quad \textup{as} \quad z \to \infty.
\end{align*}
We conclude with the following theorem:
\begin{theorem}\label{Theorem-distance-from-T-algebra}
Let $\varphi(w)= a+ \lambda w$ with $a \in \mathbb{C}$ and $\lambda \in \mathbb{C}$ with $|\lambda| \leq 1, ~\lambda \neq 1$. If $0 \ne W_{\psi, \varphi}$ is a bounded operator on $F^2$, then: 
\begin{equation*}
\textup{dist } \big{(} W_{\psi, \varphi}, \mathcal{T} \big{)} := \inf \Big{\{} \| W_{\psi, \varphi} - A \big{\|} \: : \: A \in \mathcal{T} \Big{\}} 
\geq \frac{1}{\|W_{\psi, \varphi}\|} \limsup_{|z| \rightarrow \infty} M_z(\psi, \varphi).  
\end{equation*}
\end{theorem}
\begin{proof}
The statement immediately follows by combining \eqref{GL_estimate-difference_TO_and_weighted_composition_operator} and \eqref{Estimate-Toeplitz-weighted-comp}, as well as the fact that sufficiently localized operators are dense in $\mathcal T$. 
\end{proof} 
\noindent
\begin{corollary}
Let $\varphi(w)= a+\lambda w$ with $|\lambda| \leq 1, ~\lambda \neq 1$ and $a \in \mathbb{C}$. Assume that $W_{\psi, \varphi}$ is bounded on $F^2$. If $W_{\psi, \varphi}$ is not compact, then $W_{\psi, \varphi} \notin \mathcal{T}$.
\end{corollary} 
\begin{proof}
Since  $\limsup_{|z| \rightarrow \infty} M_z(\psi, \varphi)>0$ if $W_{\psi, \varphi}$ is not compact (according to Theorem \ref{Theorem-properties-weighted-composition-operators}) 
the statement is a direct consequence of Theorem \ref{Theorem-distance-from-T-algebra}. 
\end{proof}
Combining the previous corollary with Proposition \ref{prop:compinT}, we therefore arrive at the full characterization of all weighted composition operators in $\mathcal T$:
\begin{corollary}\label{cor:weighted composition op in Toeplitz algebra}
    Let $W_{\psi, \varphi} \in \mathcal T$. Then, $W_{\psi, \varphi}$ is either compact or a constant multiple of a Weyl operator.
\end{corollary}
\begin{question}
Since every Weyl operator is a Toeplitz operator, the only question left open by the above result is the following: When is a compact weighted composition operator $W_{\psi, \varphi}$ a Toeplitz operator?
\end{question}
We now discuss a sufficient criterion to this problem. To this end, we calculate the heat transform of the Berezin transform of $W_{\psi, \varphi}$. Let $t>0$ and consider: 
\begin{align*}
\Big{(} \widetilde{W_{\psi, \varphi}} \Big{)}^{\sim (t)}(w):
&= \frac{1}{\pi t} \int_{\mathbb{C}^n} \widetilde{W_{\psi, \varphi}}(z) e^{-\frac{|z-w|^2}{t}} dw\\
&= \frac{1}{\pi t} e^{-\frac{|w|^2}{t}} \int_{\mathbb{C}} e^{-(1-\lambda+\frac{1}{t})|z|^2+ a \overline{z} + \frac{z\overline{w}}{t} + \frac{\overline{z}w}{t}} \psi(z) dz=(*).
\end{align*}
Put $\rho_t:= (1-\lambda+\frac{1}{t})$ and for the moment we assume that $\lambda$ is real and $\rho_t>0$. By the transformation rule for the integral: 
\begin{align*}
(*)
&= \frac{1}{\pi \rho_t t} e^{-\frac{|w|^2}{t}} \int_{\mathbb{C}} e^{-|z|^2+ \frac{1}{t\sqrt{\rho_t}}(\overline{z}(ta+w)+ z \overline{w})}\psi \left( \frac{z}{\sqrt{\rho_t}}\right) dz\\
&=\frac{1}{\rho_t t} e^{-\frac{|w|^2}{t}+ \frac{\overline{w}(ta+w)}{t^2\rho_t}} \psi \left( \frac{ta+w}{t \rho_t} \right) \\
&=\frac{1}{t(1-\lambda)+1} e^{- \frac{1- \lambda}{t(1-\lambda)+1}|w|^2+ \frac{a \overline{w}}{t(1-\lambda)+1}}\psi \left( \frac{ta+w}{t(1-\lambda)+1}\right). 
\end{align*}
Since this expression is analytic in $\lambda$, we conclude by the identity theorem for holomorphic functions that the equality
\begin{equation}\label{GL_equality_Berezin_transform}
\Big{(} \widetilde{W_{\psi, \varphi}} \Big{)}^{(t)}(w)=\frac{1}{t(1-\lambda)+1} e^{- \frac{1- \lambda}{t(1-\lambda)+1}|w|^2+\frac{a \overline{w}}{t(1-\lambda)+1}}\psi \left( \frac{ta+w}{t(1-\lambda)+1}\right) 
\end{equation}
extends to the case $\lambda \in \mathbb{C}$ and $|\lambda|<1$ as long as $t>0$, at least for small values of $t$. On the right hand side and in 
the case $\lambda \ne 0$ we can (formally) choose $t=-1$: 
\begin{equation*}
\Big{(} \widetilde{W_{\psi, \varphi}} \Big{)}^{(-1)}(w):=\frac{1}{\lambda} e^{\frac{\lambda-1}{\lambda}|w|^2+ \frac{a \overline{w}}{\lambda}} \psi \left( \frac{w-a}{\lambda} \right)=:f_{\psi, \varphi}(w). 
\end{equation*}
Assume now that 
\begin{equation}\label{GL_equality_Berezin_transform-2}
\textup{Re} \left(\frac{\lambda-1}{\lambda}\right)\leq 0 \hspace{4ex} \Longleftrightarrow \hspace{4ex}  \textup{Re}(\lambda) \geq |\lambda|^2 
\end{equation}
and $\psi$ is such that the right hand side of (\ref{GL_equality_Berezin_transform-2}) is bounded (e.g., $\psi$ is a holomorphic polynomial). Then we have 
\begin{equation*}
W_{\psi, \varphi}= T_{f_{\psi, \varphi}} 
\end{equation*}
since the Berezin transforms of both sides coincide. Indeed, the equality
\begin{align*}
    \widetilde{f_{\psi, \varphi}}(z) = e^{(\lambda -1)|z|^2 + a\overline{z}} \psi(z)
\end{align*}
is readily verified for real $\lambda$, followed again by an application of the identity theorem. Now, observing that the condition $\operatorname{Re}(\lambda) \geq |\lambda|^2$ describes exactly the closed disc of all $\lambda \in \mathbb C$ such that $|\lambda-\frac{1}{2}| \leq \frac{1}{2}$, we have proven:
\begin{theorem}\label{thm:compequalstoeplitz}
    Assume that $\varphi(z) = a+\lambda z$ with $|\lambda - \frac{1}{2}| \leq \frac{1}{2}, \lambda \neq 0, 1$ such that
    \begin{align*}
        f_{\psi, \varphi}(w) = \frac{1}{\lambda} e^{\frac{\lambda-1}{\lambda}|w|^2+ \frac{a \overline{w}}{\lambda}} \psi \left( \frac{w-a}{\lambda} \right)
    \end{align*}
    is a bounded function. Then, $W_{\psi, \varphi} = T_{f_{\psi, \varphi}}$.
\end{theorem}
In general, compactness of $W_{\psi, \varphi}$ is not characterized by the property that the Berezin transform vanishes at infinity. From the above theorem, we obtain the following consequence:
\begin{corollary}
    Assume that $\varphi(z) = a+\lambda z$ with $|\lambda - \frac{1}{2}| \leq \frac{1}{2}, \lambda \neq 0, 1$ and $\psi \in F^2$ such that
    \begin{align*}
        e^{\frac{\lambda-1}{\lambda}|w|^2+ \frac{a \overline{w}}{\lambda}} \psi \left( \frac{w-a}{\lambda} \right)
    \end{align*}
    is a bounded function of $w$. Then, $W_{\psi, \varphi}$ is compact if and only if $\widetilde{W_{\psi, \varphi}}(z) \to 0$ as $z \to \infty$.
\end{corollary}
\begin{example}
In particular, consider the case $\varphi(z)=a+z$ and $\psi(w)=e^{-w \overline{a}- \frac{|a|^2}{2}}=k_{-a}(w)$. Then $\lambda=1$ and we have: 
\begin{equation*}
f_{1, \varphi}(w)=e^{2 i \textup{Im}(a \overline{w})+ \frac{|a|^2}{2}}
\end{equation*}
and we recover the formula (\ref{Weyl-operator-equal-TO}) with $z=-a$. 
\end{example}
\begin{remark}
    \begin{enumerate}
        \item Indeed, the assumption that $\lambda \neq 0$ in Theorem \ref{thm:compequalstoeplitz} is no serious restriction: If $\lambda = 0$, then $W_{\psi, \varphi}$ is a rank one operator. From \cite[Corollary 3.2]{Borichev_Rozenblum2015}, we know that no Toeplitz operator with a bounded symbol can be of finite rank. So in this case, $W_{\psi, \varphi}$ can never be a Toeplitz operator.
        \item The argument used to derive Theorem \ref{thm:compequalstoeplitz} crucially hinged on the fact that the Berezin transform is injective. Nevertheless, the Berezin transform is injective on a larger class of functions. If we assume, say, that $f_{\psi, \varphi}$ is slowly increasing, i.e., it is a tempered distribution, then injectivity of the Berezin transform can still be utilized to show that $W_{\psi, \varphi} = T_{f_{\psi, \varphi}}$. Of course, there is a problem here, since $T_{f_{\psi, \varphi}}$ in general need not to be bounded for such symbols. But if we assume that $f_{\psi, \varphi} \in \mathcal S'$ such that $\widetilde{f_{\psi, \varphi}}^{(t)} \in L^\infty(\mathbb C)$ for some $0 < t < \frac{1}{2}$, then Theorem \ref{thm:BergerCoburnToeplitzAlg} ensures that the Toeplitz operator is bounded and even in the Toeplitz algebra, and the equality $W_{\psi, \varphi} = T_{f_{\psi, \varphi}}$ still holds true. Note that this approach gives rise to Toeplitz operators with unbounded symbols which are always compact.
        \end{enumerate}
\end{remark}

A natural problem which has attracted attention in the study of weighted composition operators on functions spaces is the study of their differences (see \cite{Tien_Khoi2019,Choe_Izuchi_Koo2010} and the references therein).

In the case of the Fock space, in \cite{Tien_Khoi2019} the authors showed that a difference of two weighted composition operators $W_{\psi_1,\varphi_1}-W_{\psi_2,\varphi_2}$ is bounded if and only if both operators are bounded, and compact if and only if both operators are compact. Using this characterization and our above results, we can determine the conditions under which such differences belong to $\mathcal{T}$:

\begin{corollary}
    Let  $\psi_j,\varphi_j\in\operatorname{Hol}(\mathbb C)$, $j=1,2$, with $0 \neq \psi_1 \neq \psi_2 \neq 0$. Then, $W_{\psi_1,\varphi_1}-W_{\psi_2,\varphi_2}\in \mathcal T$ if and only if $W_{\psi_j,\varphi_j}\in \mathcal T$, $j=1,2$.
\end{corollary}
\begin{proof}
    If $S\coloneqq W_{\psi_1,\varphi_1}-W_{\psi_2,\varphi_2} \in \mathcal T$, then $S$ is in particular bounded. By \cite[Theorem 1.4]{Tien_Khoi2019}, both operators $W_{\psi_j,\varphi_j}$ are bounded. Then, by Theorem \ref{Theorem-properties-weighted-composition-operators}, we have $\psi_j\in F^2$ and $\varphi_j(z)=a_j+\lambda_j z$, for all $z\in\mathbb C$, where $|\lambda_j|\leq 1$, $j=1,2$. Note that if $\lambda_j=1$ for some $j$, then one of the operators is the multiple of a Weyl operator (hence, a Toeplitz operator), and the result follows. Hence, we may assume that $\lambda_j\neq 1$ for $j=1,2$.
    
    Let now $A$ be a sufficiently localized operator. For all $z\in\mathbb C$ one can imitate the computations in \eqref{GL_estimate-difference_TO_and_weighted_composition_operator} replacing the operator $W_{\psi,\varphi}$ by $S$, from which one gets:
    \begin{align}\label{eq:ineq difference weighted composition op Berezin trans}
    \| S\| \| A - S\| \geq \left |\langle A^\ast k_z, S^\ast k_z\rangle - \langle S^\ast k_z, S^\ast k_z\rangle \right |
    =\left |\langle A^\ast k_z, S^\ast k_z\rangle - \widetilde{SS^\ast}(z) \right |.
    \end{align}
    Similarly as before, by \eqref{Estimate-Toeplitz-weighted-comp} we get
    \begin{align*}
        |\langle A^\ast k_z, S^\ast k_z\rangle|  \leq  |\langle A^\ast k_z, W_{\psi_1, \varphi_2}^\ast k_z\rangle| + |\langle A^\ast k_z, W_{\psi_2, \varphi_2}^\ast k_z\rangle|\to 0,\quad |z| \to \infty.
    \end{align*}
    Since the set of sufficiently localized operators is dense in $\mathcal T$, we can make the left-hand side of \eqref{eq:ineq difference weighted composition op Berezin trans} arbitrarily small, from which we get $\widetilde{SS^\ast}(z)\to 0$, $|z|\to\infty$. Thus, by Theorem \ref{thm:compactness}, we conclude that $S$ is compact.

    Finally, \cite[Theorem 1.4]{Tien_Khoi2019} implies that $W_{\psi_j\varphi_j}$ is compact $j=1,2$, from which the conclusion follows.
\end{proof}

\subsection{Singular integral operators}\label{sec:singularint}
For $\varphi \in F^2$ we define the \emph{singular integral operator} $S_\varphi$ by
\begin{align}\label{eq:singular integral operator definition}
    S_\varphi f(z) = \int_{\mathbb C} f(w) e^{z \cdot \overline w} \varphi(z-\overline{w})~d\mu(w).
\end{align}
This class of operators was introduced by K.\ Zhu in \cite{Zhu2015}. In \cite{Cao_Li2020}, the following result was obtained:
\begin{theorem}\label{thm:singularintegral}
Let $\varphi \in F^2$.
\begin{enumerate}
    \item $S_\varphi$ is bounded if and only if there exists $m \in L^\infty(\mathbb R)$ such that
    \begin{align}
    \label{eq:varphi sing op expression}
        \varphi(z) = \left( \frac{2}{\pi} \right)^{1/2}\int_{\mathbb R} m(x) e^{-2(x - \frac{i}{2}z)^2}~dx.
    \end{align}
    Then, $\| S_\varphi\| = \| m\|_\infty$.
    \item If $S_\varphi$ is compact, then $\varphi = 0$.
\end{enumerate}
\end{theorem}

We go on computing the bivariate Berezin transform of $S_\varphi$, assuming the operator is bounded:
\begin{align}
    \widetilde{S_\varphi}(w,z) &= e^{-\frac{|z|^2+|w|^2}{2}} S_\varphi(K_w)(z)\notag \\
    &= e^{-\frac{|z|^2 + |w|^2}{2}} \int_{\mathbb C} \varphi(z - \overline{u}) e^{z\cdot \overline{u}+ u \cdot \overline{w}}~d\mu(u)\notag \\
    &= e^{-\frac{|z|^2 + |w|^2}{2}} \varphi(z - \overline{w}) e^{z \cdot \overline{w}}.
\end{align}
Thus, we get
\begin{align}
    \widetilde{S_\varphi}(z) &= \varphi(z - \overline{z})\\
    K_{S_\varphi}(w, z) &= \varphi(z - \overline{w}) e^{z \cdot \overline{w}}.
\end{align}

We now describe how $S_\varphi$ transforms under the operator shift:
\begin{lemma}\label{lemma:singularintegralshift}
    Let $\varphi \in F^2$ such that $S_\varphi$ is bounded. Then, 
    \begin{align*}
        \alpha_v(S_\varphi) = S_{\varphi(\cdot - 2i\operatorname{Im}(v))}.
    \end{align*}
\end{lemma}
\begin{proof}
    The statement follows from straightforward calculations. We provide them for completeness.
    \begin{align*}
        \alpha_v(S_\varphi) f(z) &= e^{z \overline{v} - \frac{|v|^2}{2}} \int_{\mathbb C} W_{-v}f (w) e^{(z-v)\overline{w}} \varphi(z-v-\overline{w})~d\mu(w)\\
        &= e^{z \overline{v} - |v|^2}  \int_{\mathbb C} f(w+v) e^{-w\overline{v} + (z-v)\overline{w}} \varphi(z-v-\overline{w})~d\mu(w).
        \intertext{Substituting $w \mapsto w-v$ yields:}
        \alpha_v(S_\varphi) f(z) &= \int_{\mathbb C} f(w) e^{z\overline{w}} \varphi(z - w - 2i\operatorname{Im}(v))~d\mu(w)\\
        &= S_{\varphi(\cdot - 2i\operatorname{Im}(v))}f(z)\qedhere
    \end{align*}
\end{proof}

We remark that one of the essential tools in the derivation of Theorem \ref{thm:singularintegral} was the use of the well-known \emph{Bargmann transform} $B\colon L^2(\mathbb R) \to F^2$. Recall that $B$ is unitary and defined as the integral transform:
\begin{align}
\label{eq:Bargmann transform}
        Bf(z)=\left(\frac{2}{\pi}\right)^{\frac{1}{4}} \int_{\mathbb R} f(x) e^{2x\cdot z-x^2-\frac{z^2}{2}}dx, \quad z\in \mathbb C, f\in L^2(\mathbb R).
\end{align}
As was shown in \cite{Cao_Li2020}, a bounded singular integral operator $S_\varphi$ can be written as
\begin{equation}
\label{eq:sing int op multiplier}
        S_\varphi= B \mathcal F^{-1}M_m\mathcal F B^\ast,
\end{equation}
where $M_m$ is the multiplication operator whose symbol $m\in L^\infty(\mathbb R)$ satisfies \eqref{eq:varphi sing op expression} and $\mathcal F$ is the Fourier transform
\begin{align*}
        \mathcal Ff(x)=\pi^{-\frac{1}{2}}\int_{\mathbb R} e^{-2ix\cdot y} f(y) dy,\quad f \in L^1(\mathbb R) \cap L^2(\mathbb R), ~x\in\mathbb R.
\end{align*}
It follows easily from \eqref{eq:sing int op multiplier}, that the set $\{S_\varphi\in \mathcal L(F^2)\colon \varphi\in F^2\}$ is a commutative von Neumann algebra isomorphic to the von Neumann algebra $\{M_m\colon m\in L^\infty(\mathbb R)\}$. In fact, this observation leads to a close connection with the theory of commutative Toeplitz $C^\ast$-algebras. For this, we recall some notions from previous works.

Consider a function $a \in L^\infty(\mathbb C)$. We say that $a$ is \emph{horizontal} if $a(z+ih)=a(z)$ for all $h\in\mathbb R, z\in \mathbb C$ and \emph{vertical} if $a(z+h)=a(z)$ for all $h\in \mathbb R, z\in \mathbb C$.
Note that the symbol $a$ is horizontal if and only if $a(i\ \cdot \ )$ is vertical, and that every vertical symbol is of this form.

Similarly, we say that an operator $A\in \mathcal L(F^2)$ is \emph{horizontal} if $\alpha_{ih}(A)=A$ for all $h\in\mathbb R$ and \emph{vertical} if $\alpha_{h}(A)=A$ for all $h\in \mathbb R$. Lemma \ref{lemma:singularintegralshift} implies that every bounded singular integral operator $S_\varphi$ is vertical.

In \cite{Esmeral_Vasilevski2016} the authors study Toeplitz operators with horizontal symbols. They show that these operators are normal and mutually commuting. Specifically, they construct a unitary operator
\begin{equation*}
U\colon F^2(\mathbb C)\to L^2(\mathbb R)    
\end{equation*}
such that for any horizontal symbol $a\in L^\infty(\mathbb R)$ it holds $ T_a =U^\ast M_{\gamma_a} U$, where $M_{\gamma_a}$ is the multiplication operator with symbol given by
\begin{equation}
\label{eq:gamma horizontal Toeplitz op}
        \gamma_a(x)=\pi^{-1/2}\int_{\mathbb R}a\left( \frac{y}{\sqrt{2}}\right)e^{-(x-y)^2}dy,\quad x\in\mathbb R.
\end{equation}
We mention that the construction of the operator $U$ follows a general scheme which has been successful in the diagonalization of other types of Toeplitz operators defined on several functions spaces whose symbols are invariant under suitable group actions (see, for example, \cite{Fulsche_Rodriguez_2023,Quiroga_Vasilevski2007}).

Roughly speaking, this method consists in deforming the space of functions on which the operators act ($F^2$ in our case) in a way that allows to apply a suitable Fourier transform along the variable under which the symbols are invariant. After some technical details, the composition of these transformations turn out to diagonalize the invariant Toeplitz operators under consideration.

Interestingly, in the case of horizontal symbols the above unitary operator $U$, obtained through this construction, coincides with the inverse of the Bargmann transform \eqref{eq:Bargmann transform} up to some rescaling. More precisely, given $t>0$, consider the unitary operator $\mathcal R_t$ on $L^2(\mathbb R)$ defined by
\begin{align*}
        \mathcal R_tf(x)=t^{1/2}f(t x),\quad x\in \mathbb R, \ f \in L^2(\mathbb R).
\end{align*}
By simple computations, one can check that $\mathcal R_t^*=\mathcal R_t^{-1}=\mathcal R_{t^{-1}}$ and $\mathcal R_t^\ast M_m \mathcal R_t=M_{m(t^{-1}\cdot)}$ for all $m\in L^\infty(\mathbb R)$.
Then, comparing the expressions in \cite{Esmeral_Vasilevski2016} with \eqref{eq:Bargmann transform}, we get $B=U^\ast \mathcal R_{2^{-\frac{1}{2}}}$.

Let now $m\in L^\infty (\mathbb R)$. We denote by $\varphi_m\in F^2$ its associated function given by \eqref{eq:varphi sing op expression} and by $S[\varphi_m]$ the induced (bounded) singular integral operator. Then we obtain from \eqref{eq:sing int op multiplier}:
\begin{align}
\label{eq:sing int op rescaling}
        S[\varphi_m]&=B\mathcal{F}^{-1} M_m \mathcal{F} B^\ast \nonumber\\
        &=(B\mathcal F^{-1} B^\ast) U^\ast \mathcal R_{2^{-\frac{1}{2}}} M_{m}  \mathcal R_{2^{\frac{1}{2}}} U (B \mathcal F B^\ast) \nonumber \\
        &=(B\mathcal F^{-1} B^\ast) U^\ast M_{m(2^{-\frac{1}{2}}\ \cdot\ )} U (B \mathcal F B^\ast).  
\end{align}

We remark that in previous works \cite{Bais_Venku2023,Cao_Li2020} it was noted that the set of bounded singular integral operators forms a maximal commutative $C^*$-algebra. Here, we give a more precise characterization:

\begin{theorem}
        The set of bounded singular integral operators $\mathcal S\coloneqq \{S_\varphi\in\mathcal L(F^2)\colon \varphi\in F^2\}$ coincides with the von Neumann algebra of vertical operators 
        \begin{equation*}
        \mathcal A_{\mathbb R}\ \coloneqq \{A\in \mathcal L(F^2)\colon \alpha_h(A)=A,\ \text{for all }h\in \mathbb R\}.
\end{equation*}

\end{theorem}
\begin{proof}
    As a consequence of Lemma \ref{lemma:singularintegralshift}, every bounded singular integral operator is vertical, hence $\mathcal S\subset \mathcal A_{\mathbb R}$.

    Let now $a\in L^\infty(\mathbb R)$ and let $S[\varphi_{\gamma_a(2^{-\frac{1}{2}}\ \cdot\ )}]$ be the singular integral operator associated to the function $m=\gamma_a(2^{-\frac{1}{2}}\ \cdot\ )\in L^\infty(\mathbb R)$, where $\gamma_a$ is given by \eqref{eq:gamma horizontal Toeplitz op}. As is well-known \cite{Zhu2019, Cao_Li2020}, the unitary operator $\mathcal F_B\coloneqq B\mathcal F B^{-1}\in \mathcal L(F^2)$ satisfies $\mathcal F_Bf(z)=f(-iz)$ and $\mathcal F^{-1}_Bf(z)=f(iz)$, for all $z\in\mathbb C$ and $f\in F^2$. Thus, by a standard computation, we obtain
\begin{equation}
\label{eq:singular op Toeplitz vertical}
        S[\varphi_{\gamma_a(2^{\frac{1}{2}}\ \cdot\ )}]=\mathcal F_B^{-1} U^\ast M_{\gamma_a} U \mathcal F_B=\mathcal F_B^{-1} T_a \mathcal F_B=T_{a(i \ \cdot \ )}.
\end{equation}
Since every vertical symbol can be written as $a(i \cdot \ )$ for some horizontal symbol $a$, it follows from \eqref{eq:singular op Toeplitz vertical} that
\begin{equation*}
        \{T_a\in \mathcal L(F^2)\colon a\text{ \it is vertical}\}\subset \mathcal S.
\end{equation*}
By \cite[Corollary 3.4]{Fulsche_Rodriguez_2023}, the algebra $\mathcal A_{\mathbb R}$ is generated (as a von Neumann algebra) by the Weyl operators contained in it. Since any such Weyl operator is also a Toeplitz operator with vertical symbol, it follows that $\mathcal S=\mathcal A_{\mathbb R}$.
\end{proof}

\begin{remark}
We remark that the above results extend naturally to the Fock space $F^2(\mathbb C^n)$ of Gaussian square integrable entire functions on $\mathbb C^n$. Moreover, they hold in a more general form replacing horizontal or vertical symbols and operators by objects invariant under translation in the direction of a Lagrangian subspace $\mathcal L \subset \mathbb C^n$.

Specifically, in \cite{Bais_Venku2023}, the authors consider the singular integral operator
\begin{align*}
    S^X_\varphi f(z) = \int_{\mathbb C^n} f(w) e^{z\cdot \overline{w}} \varphi(w-X^*\overline{X w})d\mu(w),
\end{align*}
where $X\in U(n,\mathbb C)$ is a unitary matrix and $\varphi \in F^2(\mathbb C^n)$.

Given such a matrix $X$ one can consider the unitary operator $U_X\colon F^2(\mathbb C^n)\to F^2(\mathbb C^n)$ defined by $U_Xf(z)=f(X^* z)$. Then it holds
\begin{align*}
    U_X S_{\varphi}^I U_X^*=S^X_\varphi, 
\end{align*}
where $S_{\varphi}^I$ is the singular integral operator
\begin{align*}
    S_{\varphi}^I f(z) = \int_{\mathbb C^n} f(w) e^{z\cdot \overline{w}} \varphi(z-\overline{w}) d\mu (w).
\end{align*}
One can see that $S_{\varphi}^I$ is a vertical operator, meaning that $W_{h}S_{\varphi}^IW_{-h}=S_{\varphi}^I$, for all $h\in \mathbb R^n\times \{0\}\subset \mathbb C^n$, where $W_{w}$ denotes the Weyl operator defined analogously on $F^2(\mathbb C^n)$ by
$W_{w}f(z)=e^{z\cdot \overline{w}-\frac{|w|^2}{2}}f(z-w)$, $z,w\in \mathbb C^n$, $f\in F^2(\mathbb C^n)$.

Theorem \ref{thm:singularintegral} holds analogously on the Fock space $F^2(\mathbb C^n)$ (cf.\ \cite{Cao_Li2020}). Hence, by similar arguments as above, one can prove that the set of bounded singular integral operators
\begin{align*}
    \{S^X_\varphi\in \mathcal L(F^2(\mathbb C^n))\colon \varphi \in F^2(\mathbb C^n)\}
\end{align*}
coincides with the von Neumann algebra of $X(\mathbb R^n\times \{0\})$-invariant operators:
\begin{align*}
    \left \{ A\in \mathcal L(F^2(\mathbb C^n)) \colon W_{Xh}AW_{-Xh}=A,\forall h\in \mathbb R^n\times \{0\} \right \}.
\end{align*}
\end{remark}

We add two more facts to the discussion. Of course, analogous statements also hold for the singular integral operators with respect to a Lagrangian subspace that were discussed above.
\begin{theorem}\label{thm_singular-integral-operators-questions-A-B}
    Let $\varphi \in F^2$ and $m \in L^\infty(\mathbb R)$ be related by Eq.\ \eqref{eq:varphi sing op expression}.
    \begin{enumerate}
        \item $S_\varphi \in \mathcal T$ if and only if $m \in \operatorname{BUC}(\mathbb R)$. 
        \item $S_\varphi = T_f$ for some vertical symbol $f \in L^\infty(\mathbb C)$ if and only if one can factorize $m = g \ast m_0$ with $m_0 \in L^\infty(\mathbb R)$, where $g(x) = \sqrt{2/\pi} e^{-2x^2}$. Then, $f(z) = m_0(\operatorname{Im}(z))$. 
    \end{enumerate}
\end{theorem}

\begin{proof}
We first want to emphasize that there are several possible ways to prove these statements. Besides the arguments presented below, the equations \eqref{eq:gamma horizontal Toeplitz op} and \eqref{eq:sing int op rescaling} can also be used for a proof.
\begin{enumerate}
    \item From the formula relating $\varphi$ and $m$ from Theorem \ref{thm:singularintegral} we obtain for $y \in \mathbb R$:
    \begin{align*}
        \varphi(z - iy) &= \left( \frac{2}{\pi}\right)^{1/2} \int_{\mathbb R} m(x) e^{-2(x - \frac{i}{2}(z-iy))^2}~dx\\
        &= \left( \frac{2}{\pi}\right)^{1/2} \int_{\mathbb R} m(x + \frac{y}{2}) e^{-2(x - \frac{i}{2}z)^2}~dx.
    \end{align*}
    Such that, using also Lemma \ref{lemma:singularintegralshift}:
    \begin{align*}
        \| S_{\varphi} - \alpha_v(S_\varphi)\| &= \| S_{\varphi} - S_{\varphi(\cdot - 2i\operatorname{Im}(v))}\|\\
        &= \| S_{\varphi - \varphi(\cdot - 2i\operatorname{Im}(v))}\|\\
        &= \| m - m(\cdot + \operatorname{Im}(v))\|_\infty.
    \end{align*}
    Clearly, this converges to zero as $v \to 0$ if and only if $m \in \operatorname{BUC}(\mathbb R)$. Hence, the statement follows from Theorem \ref{thm:Toeplitz_eq_c1}.
    \item Let $f \in L^\infty(\mathbb C)$ be a vertical symbol such that $f(z) = f(z+x)$ for all $z \in \mathbb C$, $x \in \mathbb R$. Then, we have the following for the Berezin transform of $T_f$ (writing $z = a+ib$ and $w = x+iy$):
    \begin{align*}
        \widetilde{T_f}(z) &= e^{-|z|^2} \frac{1}{\pi} \int_{\mathbb C} f(w) e^{z\overline{w} + w\overline{z}} e^{-|w|^2}~dw\\
        &= e^{-a^2 - b^2} \frac{1}{\pi} \int_{\mathbb R}  \int_{\mathbb R} f(iy)e^{(a+ib)(x-iy) + (x+iy)(a-ib)} e^{-x^2 - y^2}~dy~dx\\
        &= \frac{1}{\pi} \int_{\mathbb R} e^{-(x-a)^2}~dx \int_{\mathbb R} f(iy) e^{-(y-b)^2}~dy\\
        &= f|_{i\mathbb R} \ast g_0(b),
    \end{align*}
    where $g_0(t) = \frac{1}{\sqrt{\pi}}e^{-t^2}$. On the other hand, we know that
    \begin{align*}
        \widetilde{S_\varphi}(a+ib)= \varphi(2ib)
    \end{align*}
    such that 
    \begin{align*}
        \varphi(2ib) = m \ast g (-b),
    \end{align*}
    where $g(t) = \sqrt{2/\pi}e^{-2t^2}$. Note that $g \ast g = g_0$. If $m = m_0 \ast g$, then we see that with $f(z) = m_0(-\operatorname{Im}(z))$ we have
    \begin{align*}
        \widetilde{f}(a+ib) = m_0 \ast g(-b) = m_0 \ast g_0 \ast g_0(-b) = \widetilde{S_{\varphi}}(z)
    \end{align*}
    such that uniqueness of the Berezin transform implies $T_f = S_\varphi$. On the other hand, if we assume that $T_f = S_\varphi$, then we have 
    \begin{align*}
        m \ast g (-\cdot) = f|_{i\mathbb R} \ast g_0.
    \end{align*}
    Since convolving by $g_0$ or $g$ is injective, $m_0(-y) = f(iy)$ solves the equation $m = m_0 \ast g_0$.
\end{enumerate}
\end{proof}

To finish this section, we remark that the theory of vertical or horizontal operators, and more generally, the theory of operators invariant under translations by Lagrangian subspaces, forms part of a more general framework developed in \cite{Fulsche_Rodriguez_2023} using \emph{quantum harmonic analysis}. In particular, there is a discrete analogue of these algebras given as follows.

Consider the \emph{von Neumann lattice} $G=\sqrt{\pi}\mathbb Z + i \sqrt{\pi}\mathbb Z \subset  \mathbb C$. In \cite{Fulsche_Rodriguez_2023} the authors showed that the von Neumann algebra
\begin{align*}
    \mathcal A_G=  \left \{ A\in \mathcal L(F^2)\colon \alpha_w(A)=A,\ \forall w\in G \right \}
\end{align*}
is commutative. Moreover, they constructed a Bargmann-type transform diagonalizing these operators as in the case of horizontal operators.

We will write $w_d = \sqrt{\pi} d_1 + i \sqrt{\pi} d_2$, where $d = (d_1, d_2) \in \mathbb Z^2,$ and consider the operator
\begin{align*}
    V & \colon L^2(\mathbb C,\mu) \longrightarrow L^2(R_0,\mu)\otimes L^2(\mathbb T^2,\nu),\\
    V(f)(z, \lambda) & := \sum_{d \in \mathbb Z^2} e^{-z \overline{w_d} - \frac{|w_d|^2}{2}} f(z+w_d)\lambda^{-d}, \quad (z, \lambda) \in R_0 \times \mathbb T^2, \ f\in F^2,
\end{align*}
where $R_0=[0,\sqrt{\pi}]^2$, $\mu$ is the Gaussian measure (restricted to $R_0$ on the right-hand side), $\mathbb T^2$ is the $2$-torus and $\nu$ is its normalized Haar measure.
One can show that this operator is unitary.

Further, consider the measure $\eta$ on $\mathbb T^2$ given by $d\eta= H d\nu$, $\lambda\in\mathbb T^2$, where the nowhere vanishing function $H$ is given by
\begin{align*}
H(\lambda)=\int_{R_0}|V(1)(z,\lambda)|^2 d\mu(z),\quad \lambda\in\mathbb T^2,    
\end{align*}
and let $S$ be the isometric operator
\begin{align*}
     S\colon L^2(\mathbb T^2,\eta) & \longrightarrow L^2(R_0,\mu)\otimes L^2(\mathbb T^2,\nu),\\
     f & \longmapsto f \cdot V(1).
\end{align*}
Using these operators we can construct an analogue of the Bargmann transform \eqref{eq:Bargmann transform}:
\begin{proposition}[{\cite{Fulsche_Rodriguez_2023}}]
The operator
    \begin{align*}
    B_G=V^\ast S \colon L^2(\mathbb T^2,\eta) \longrightarrow F^2,
    \end{align*}
is unitary and maps the function $\lambda \mapsto \lambda^d$ to the normalized reproducing kernel $k_{w_d}$, for all $d \in\mathbb Z^2$.

Moreover, for each operator $A\in \mathcal A_G$ it holds $ B_G^\ast A B_G=M_{\gamma_A}$,
where $M_{\gamma_A}$ denotes the multiplication operator acting on $L^2(\mathbb T^2)$ with symbol $\gamma_A \in L^\infty(\mathbb T^2)$ and the map $A\mapsto \gamma_A$ is a $*$-isomorphism.
\end{proposition}

Similarly to the case of the singular integral operators defined \eqref{eq:singular integral operator definition}, where $G$ was instead the Lagrangian subspace $\mathbb R \oplus \{ 0\}$, we can now construct an ``singular integral operator’’ using \eqref{eq:sing int op multiplier}. Since $\eta$ is a not a Haar measure, the most natural Fourier transform defined on $L^2(\mathbb T^2,\eta)$ is given by using the orthonormal basis $(\varepsilon_k)_{k\in\mathbb Z^2}$, where $\varepsilon_k(\lambda)=\lambda^k H(\lambda)^{-1/2}$, $\lambda\in \mathbb T^2$. Hence, we consider the operator
\begin{align*}
        \mathcal F_G \colon L^2(\mathbb T^2,\eta) &\longrightarrow l^2(\mathbb Z^2),\\
        f & \longmapsto \left( \int_{\mathbb T^2} f(\lambda) \overline{\varepsilon_{k}(\lambda)}\right)_{k\in\mathbb Z^2}.
\end{align*}
Finally, given $m\in L^\infty(\mathbb Z^2)$, consider the operator $\mathcal F_G^\ast M_m  \mathcal F_G$. By construction, this is a diagonal operator acting on the orthonormal basis of $L^2(\mathbb T,\eta)$ acting by
\begin{align*}
    \mathcal F_G^\ast M_m  \mathcal F_G(\varepsilon_k)= m(k) \varepsilon_k,\quad k\in \mathbb Z^2.
\end{align*}
Then we can define the operator
\begin{align}\label{eq:von Neumann singular operator}
S_G[m]=B_G \mathcal F_G^\ast M_m  \mathcal F_G B_G^\ast \colon F^2\longrightarrow F^2.
\end{align}
This operator is clearly diagonal with respect to the orthonormal basis $(V^\ast S(\varepsilon_k))_{k\in\mathbb Z^2}$. In particular, it is compact if and only if $m\in c_0(\mathbb Z^2)$. However, it is unclear under which conditions such an operator belongs to the Toeplitz algebra. Moreover, a closed expression for the functions $\varepsilon_k$ is unknown.

\begin{question}
Let $m\in L^\infty(\mathbb Z^2)$ and let $S_G[m]$ be the operator given by \eqref{eq:von Neumann singular operator}. Is there a closed expression for this operator analogous to \eqref{eq:singular integral operator definition}? When does the operator $S_G[m]$ belong to $\mathcal T$?
\end{question}

\subsection{Generalized Volterra-type operators}\label{sec:volterra}
One of the best known integral operators is the Volterra operator, which sends a function $f\in L^2[0,1]$ to its indefinite integral $t\mapsto \int_0^tf(s)ds$.
This type of operator, along with its generalizations, has been extensively studied in various function spaces (see, for example, the survey \cite{Siskakis_2004}), finding applications particularly in the study of linear isometries (see \cite{Fleming_Jamison_2002}).

In recent years, several authors have investigated the operator-theoretic properties of such operators in the context of the Fock space. Specifically, in \cite{Fekadiea2023,Mengestie_Worku2018} the authors study Volterra-type operators of the form:
\begin{align}\label{eq:definition V(g,varphi)}
    V_{(g,\varphi)}f(z) = \int_0^z (f\circ \varphi)(w) g'(w) dw, \quad f\in \operatorname{Hol}(\mathbb C),
\end{align}
where $g$ and $\varphi$ are holomorphic functions. 

Note that letting $\varphi(z)=z$ we obtain operators of the form
\begin{align}\label{eq:definition Vg}
    V_gf(z) = \int_0^z f(w) g'(w) dw,
\end{align}
which were studied in \cite{Constantin_2012} and which are also the $1$-dimensional versions of the so-called Ces\`{a}ro operators (see \cite{Chen_Wang2022} and the references therein).
Furthermore, we note that the operator $V_z$ is the exact analogue of the classical Volterra operator.

As it turns out, and similarly as in the case of weighted composition operators, many properties of the operator $V_{(g,\varphi)}$ depend on the behaviour of the function
\begin{align}\label{eq:definition R(g,varphi)}
    R_{(g,\varphi)}(z)\coloneqq\frac{|g'(z)|}{1+|z|}e^{\frac{1}{2}(|\varphi(z)|^2-|z|^2)}.
\end{align}
For instance, the following characterization for boundedness and compactness of Volterra-type operators $V_{(g,\varphi)}$ was obtained in \cite{Mengestie_Worku2018}:
\begin{theorem}\label{thm:boundedness compactness Volterra type}
    Let $g,\varphi\in \operatorname{Hol}(\mathbb C)$ and let $V_{(g,\varphi)}$ be defined as in \eqref{eq:definition V(g,varphi)}.
    \begin{enumerate}
        \item[(1)] $V_{(g,\varphi)}$ defines a bounded operator on $F^2$ if and only if $\sup_{z\in\mathbb C}|R_{(g,\varphi)}(z)|<\infty$.

        \item[(2)] $V_{(g,\varphi)}$ defines a compact operator on $F^2$ if and only if $R_{(g,\varphi)}(z)\to 0$ as $|z|\to\infty$.
    \end{enumerate}
    In either case, there are constants $a,\lambda\in\mathbb C$ with $|\lambda|\leq 1$ such that $\varphi(z)=a+\lambda z$. If $V_{(g,\varphi)}$ is compact, then $|\lambda|<1$.
\end{theorem}
As a corollary, the following result, originally due to Constantin (see \cite{Constantin_2012}) was obtained:
\begin{corollary}\label{cor:boundedness compactness Vg}
    Let $g\in\operatorname{Hol}(\mathbb C)$ and consider the operator $V_g$ given by \eqref{eq:definition Vg}.
    \begin{enumerate}
        \item[(1)] $V_g$ defines a bounded operator on $F^2$ if and only if $g(z)=a+bz+cz^2$ for all $z\in\mathbb C$ and some constants $a,b,c\in \mathbb C$.

        \item[(2)] $V_g$ defines a compact operator on $F^2$ if and only if $g(z)=a+bz$ for all $z\in\mathbb C$ and some constants $a,b\in \mathbb C$.
    \end{enumerate}    
\end{corollary}

Next, we analyze the question of when an operator $V_{(g,\varphi)}$ belongs to $\mathcal T$. We begin with the operators $V_g$.

\begin{proposition}\label{prop:Vz2 equals Toeplitz}
        Let $g(z)=\frac{1}{2}z^2$. Then the (bounded) operator $V_g$ defined on $F^2$ by \eqref{eq:definition Vg} coincides with the Toeplitz operator $T_{\phi}$, where $\phi(z)=\left(\frac{z}{|z|}\right)^2$. In particular, $V_g\in\mathcal T$.
\end{proposition}
\begin{proof}
        This follows easily by observing the action of the operators $V_g$ and $T_\phi$ on the standard orthonormal basis of $F^2$.
        Indeed, for $n\in\mathbb N_0$ one has
        \begin{align*}
            V_g(e_n)(z)=\frac{1}{\sqrt{n!}}\int_0^z w^{n+1} ~dw
            =\frac{1}{\sqrt{n!}}\cdot \frac{z^{n+2}}{n+2}=\sqrt{\frac{n+1}{n+2}} e_{n+2}(z).
        \end{align*}
        On the other hand, integrating in polar coordinates we get for any $n,m\in\mathbb N_0$:
        \begin{align*}
            \langle T_\phi e_n, e_m\rangle & = \frac{1}{\sqrt{n!m!}}\int_{\mathbb C} \left(\frac{z}{|z|}\right)^2 z^n \overline{z^m} ~d\mu(z)\\
            &=\frac{1}{\sqrt{n!m!}}\frac{1}{\pi}\int_0^\infty \int_0^{2\pi} e^{i(n+2-m)\theta} r^{n+m+1}e^{-r^2}~d\theta dr\\
            &=\delta_{n+2,m} \frac{1}{\sqrt{n!(n+2)!}}
            \int_0^\infty r^{n+1} e^{-r} ~dr\\
            &=\delta_{n+2,m}\sqrt{\frac{n+1}{n+2}}.\qedhere
        \end{align*}
\end{proof}
By writing down the definitions of the corresponding operators in the above proposition, we obtain the identity:
\begin{align*}
        \int_{\mathbb C} \left(\frac{w}{|w|}\right)^2f(w) e^{z\overline{w}} ~d\mu(w) = \int_0^z f(w) w ~dw,\quad f\in F^2, z\in\mathbb C.
\end{align*}
Furthermore, as a by-product of the above computations, one readily sees that the adjoint $V_{z^2/2}^\ast=T_\phi^\ast=T_{\overline{\phi}}$ is given by
\begin{equation}\label{eq:Tphiadj}
        T_\phi^\ast (e_n)=
        \begin{cases}
            \sqrt{\frac{n-1}{n}} e_{n-2},& \quad n\geq 2,\\
            0,& \quad \text{otherwise}.
        \end{cases}
\end{equation}
In particular, it holds
\begin{align*}
        T_\phi^\ast T_\phi (e_n)=\frac{n+1}{n+2} e_n, \quad n\in\mathbb N_0,
\end{align*}
so that $V_{z^2/2}^\ast V_{z^2/2}=T_\phi^\ast T_\phi=I+K$, where $K$ is a compact diagonal operator.

\begin{corollary}
        Let $g\in\operatorname{Hol}(\mathbb C)$. The operator $V_g$ is bounded in $F^2$ if and only if $V_g\in\mathcal T$.
\end{corollary}
\begin{proof}
        If $V_g$ is bounded, then by Corollary \ref{cor:boundedness compactness Vg}, we have $g(z)=a+bz+cz^2$, for some constants $a,b,c\in\mathbb C$. Then we have
        $V_g=V_{a+bz}+2cT_\phi$, where the left-hand summand is a compact operator by Corollary \ref{cor:boundedness compactness Vg}. Hence, $V_g\in \mathcal T$.
\end{proof}

Moreover, using Proposition \ref{prop:Vz2 equals Toeplitz} and our previous results on weighted composition operators, we can give a complete characterization for Volterra-type operators belonging to the Toeplitz algebra:
\begin{theorem}\label{thm:volterra_toeplitzalg}
    Let $g,\varphi\in \operatorname{Hol}(\mathbb C)$. The Volterra-type operator $V_{(g,\varphi)}$ belongs to $\mathcal T$ if and only if one of the two following conditions hold:
    \begin{enumerate}
        \item[(1)] $V_{(g,\varphi)}$ is compact (and hence characterized by (2) in Theorem \ref{thm:boundedness compactness Volterra type}).

        \item[(2)] There are constants $a,b,c\in \mathbb C$ such that $g'(z)=a+bze^{z\overline{c}}$ and $\varphi(z)=z-c$. In this case, we have
            \begin{align}\label{eq:thm form Volterra operators in T}
            V_{(g,\varphi)} =  a V_{z} W_{1,\varphi}  + b e^{\frac{|c|^2}{2}}V_{z^2/2} W_c,
        \end{align}
        where $W_c$ denotes the Weyl operator associated to $c\in \mathbb C$ and the operators $V_{z}$ and $V_{z^2/2}$ are given by \eqref{eq:definition Vg}.
    \end{enumerate}
\end{theorem}
\begin{proof}
    Using Proposition \ref{prop:Vz2 equals Toeplitz}, one sees that compact operators and operators of the form \eqref{eq:thm form Volterra operators in T} belong to $\mathcal T$.

    Conversely, suppose that $V_{(g,\varphi)}\in\mathcal T$. In particular, it is bounded on $F^2$ and thus, from
    \begin{align*}
        (V_{(g,\varphi)}(1))(z)=\int_0^z g'(w) ~dw = g(z)-g(0), \quad z\in\mathbb C,
    \end{align*}
    we get that $g\in F^2$.
    Further, by Theorem \ref{thm:boundedness compactness Volterra type} we have
    $\sup_{w\in\mathbb C}|R_{(g,\varphi)}(w)|^2<\infty$ and we know that $\varphi(z)=\lambda z + a$, for some $a,\lambda\in\mathbb C$ with $|\lambda|\leq 1$.
    
Write $g'(w)=g'(0) + w h(w)$, where $h\in\operatorname{Hol}(\mathbb C)$.
Using that $g\in F^2$ and its Taylor series expansion, one can easily check that $h\in F^2$ as well. On the other hand, for $|w|$ large enough we have
\begin{align*}
        M_w(h,\varphi)=|h(w)|^2 e^{|\varphi(w)|^2-|w|^2}
        \sim
        \left(\frac{|w h(w)|}{1+|w|}\right)^2 e^{|\varphi(w)|^2-|w|^2}
        \sim
        |R_{(g,\varphi)}(w)|^2,
\end{align*}
where $M_w(h,\varphi)$ is given by \eqref{eq:mz}, from which we see that $M(h,\varphi)=\sup_{w\in\mathbb C}M_w(h,\varphi)<\infty$.
Therefore, by Theorem \ref{Theorem-properties-weighted-composition-operators} we conclude that the weighted composition operator $W_{h,\varphi}$ is bounded.

Similarly, we have
\begin{align*}
        \left(\frac{|g'(0)|}{1+|w|}\right)^2e^{|\varphi(w)|^2-|w|^2} \lesssim |R_{(g,\varphi)}(w)|^2
\end{align*}
and so $\operatorname{sup}_{w\in\mathbb C} \left(\frac{|g'(0)|}{1+|w|}\right)^2e^{|\varphi(w)|^2-|w|^2} < \infty$.
Under the condition $\varphi(z)=\lambda z + a$, this is easily seen to be equivalent to
\begin{align*}
        M(1,\varphi)=\sup_{w\in\mathbb C} e^{|\varphi(w)|^2-|w|^2} < \infty.
\end{align*}
Therefore, the weighted composition operator $W_{1,\varphi}$ is also bounded.

Observe now that $V_{(g,\varphi)}$ can be written as
\begin{align*}
    V_{(g,\varphi)} = g'(0) V_{z} W_{1,\varphi}  + V_{z^2/2} W_{h,\varphi}.
\end{align*}
By Theorem \ref{cor:boundedness compactness Vg} and the boundedness of $W_{1,\varphi}$, the first summand is compact. In particular, we have $V_{z^2/2} W_{h,\varphi}\in \mathcal T$.

Finally, since $I- V_{z^2/2}^\ast V_{z^2/2} $ is compact, as was noted before,
and $V_{z^2/2}^\ast \in \mathcal T$, we conclude that
\begin{align*}
        W_{h,\varphi} = ( I- V_{z^2/2}^\ast V_{z^2/2} ) W_{h,\varphi} + V_{z^2/2}^\ast V_{z^2/2} W_{h,\varphi} \in\mathcal T,
\end{align*}
Hence, by Corollary \ref{cor:weighted composition op in Toeplitz algebra}, the operator $W_{h,\varphi}$ is either compact or a multiple of a Weyl operator, which implies conditions (1) and (2), respectively.
\end{proof}
An immediate application of the previous result, together with Proposition \ref{prop:Vz2 equals Toeplitz}, is the following fact:
\begin{corollary}
    Assume $g, \varphi \in \operatorname{Hol}(\mathbb C)$ such that $V_{(g, \varphi)} \in \mathcal T$. If $V_{(g, \varphi)}$ is not compact, then it is Fredholm with
    \begin{align*}
        \operatorname{ind} V_{(g, \varphi)} = -2.
    \end{align*}
\end{corollary}
\begin{proof}
    As in the proof of Proposition \ref{prop:Vz2 equals Toeplitz}, we see that $V_{\frac{1}{2}z^2} = T_\phi$ is off-diagonal with respect to the standard basis:
    \begin{align*}
        T_\phi e_n = \sqrt{\frac{n+1}{n+2}} e_{n+2}, \quad n \in \mathbb N_0.
    \end{align*}
    For the adjoint operator, $V_{\frac{1}{2}z^2}^\ast = T_\phi^\ast = T_{\overline{\phi}}$, we recall from Eq.\ \eqref{eq:Tphiadj}:
    \begin{align*}
        T_{\overline{\phi}} e_n = \begin{cases}
            \sqrt{\frac{n-1}{n}}e_{n-2}, \quad &n \geq 2,\\
            0, \quad &n = 0, 1.\\
        \end{cases}
    \end{align*}
    Using these mapping properties, it is not hard to verify that 
    \begin{align*}
        \operatorname{ker}(T_\phi) = \{ 0\}, \quad \operatorname{ker}(T_{\overline{\phi}}) = \operatorname{span} \{ e_0, e_1\}.
    \end{align*}
    Hence, $\operatorname{ind}(T_\phi) = -2$. If we now assume that $V_{(g, \varphi)} \in \mathcal T$ is not compact, then by Theorem \ref{thm:volterra_toeplitzalg} we can write $V_{(g, \varphi)}$ as follows, where $b \neq 0$ and $c \in \mathbb C$:
    \begin{align*}
        V_{(g, \varphi)} = aV_z W_{1, \varphi} + be^{\frac{|c|^2}{2}} V_{z^2/2}W_c.
    \end{align*}
    Since $V_z$ is compact and $W_c$ is unitary, we see that $V_{(g, \varphi)}$ is Fredholm with
    \begin{align*}
        \operatorname{ind}(V_{(g, \varphi)}) &= \operatorname{ind}(V_{z^2/2} W_c)\\
        &= \operatorname{ind}(V_{z^2/2}) + \operatorname{ind}(W_c)\\
        &= \operatorname{ind}(T_\phi) = -2.\qedhere
    \end{align*}
\end{proof}
We end this section by computing the Berezin transform of the operator $V_{(g, \varphi)}$, assuming that it is a bounded operator. The Berezin transform seems not to be contained in the literature so far. For doing this computation, we need the following auxiliary fact:
\begin{lemma}
    For $k \in \mathbb N$ we denote by $A^{[k]}$ the linear operator obtained by extending
    \begin{align*}
        A^{[k]} e_n = \frac{1}{\sqrt{(n+1)(n+2)\dots (n+k)}} e_{n+k}
    \end{align*}
    linearly. Then, $A^{[k]}$ is compact on $F^2$ and $\| A^{[k]}\| = \frac{1}{\sqrt{k!}}$ and satisfies $\frac{\partial^k}{\partial z^k} A^{[k]} g(z) = g(z)$.
\end{lemma}
\begin{proof}
    $\| A^{[k]}\| \leq \frac{1}{\sqrt{k!}}$ is readily verified by considering the expansion of arbitrary elements of $F^2$ into the standard orthonormal basis. Further, $\| A^{[k]} e_0\| = \frac{1}{\sqrt{k!}}$, showing equality of the norm. $A^{[k]}$ is clearly compact, as
    \begin{align*}
        \frac{1}{\sqrt{(n+1)(n+2)\dots (n+k)}} \to 0, \quad n \to \infty.
    \end{align*}
    Finally, $\frac{\partial^k}{\partial z^k} A^{[k]} e_m(z) = e_m(z)$ follows immediately for every $m \in \mathbb N_0$, hence the equality follows for all $g \in F^2$. 
\end{proof}
For convenience, we will also write $A^{[0]} = I$ in the following. 
\begin{lemma}
    Let $z \in \mathbb C$ and assume that $g, \varphi \in \operatorname{Hol}(\mathbb C)$ such that $V_{(g, \varphi)}$ is bounded. Then, $V_{(g, \varphi)} K_w (z) = K_w(\varphi(z)) \sum_{k=0}^\infty (-\lambda \overline{w})^k A^{[k]} g(z)$.
\end{lemma}
\begin{proof}
    In the following, we assume that $g$ is such that $g(0) = 0$ (replacing $g$ by $g - g(0)$ does not change the operator). Recall that $\varphi(z) = a + \lambda z$ with $a, \lambda \in \mathbb C$ and $|\lambda|\leq 1$ (according to Theorem \ref{thm:boundedness compactness Volterra type}). For $w = 0$, we obtain:
    \begin{align*}
        V_{(g, \varphi)} 1(z) = \int_0^z 1 \cdot g'(u)~du = g(z).
    \end{align*}
    Similarly, for arbitrary $w \in \mathbb C$ and $\varphi(u) = a \in \mathbb C$ (i.e., $\lambda = 0$), we obtain
    \begin{align*}
        V_{(g, \varphi)} K_w(z) &= e^{a \cdot \overline{w}} g(z).
    \end{align*}
    Hence, we now assume that $w \neq 0 \neq \lambda$. Then, using integration by parts:
    \begin{align*}
        V_{(g, \varphi)} K_w(z) &= \int_0^z e^{\varphi(u) \cdot \overline{w}} g'(u) ~du\\
        &= e^{a \cdot \overline{w}}\int_0^z e^{\lambda u \cdot \overline{w}} g'(u)~du\\
        &= e^{a \cdot \overline{w}} \left( [e^{\lambda u \cdot \overline{w}} g(u)]_{u=0}^z - \lambda \overline{w}\int_0^z e^{\lambda u \cdot \overline{w}} g(u)~du  \right)\\
        &= K_w(\varphi(z)) g(z) - \lambda \overline{w} V_{(A^{[1]} g, \varphi)} K_w(z).
    \end{align*}
    By induction, one now shows that:
    \begin{align*}
        V_{(g, \varphi)} K_w(z) = K_w(\varphi(z)) \sum_{k=0}^\infty (-\lambda \overline{w})^k A^{[k]} g(z).
    \end{align*}
    Note that this series converges in $F^2$ by the norm estimate for $A^{[k]}$. 
\end{proof}
\begin{corollary}
    Assume that $g, \varphi \in \operatorname{Hol}(\mathbb C)$ such that $V_{(g, \varphi)}$ is bounded. Then, for all $z, w \in \mathbb C$ the following identities are valid:
    \begin{align}
        \widetilde{V_{(g, \varphi)}}(w,z) &= e^{-\frac{|w|^2 + |z|^2}{2} + \varphi(z) \cdot \overline{w}} \sum_{k=0}^\infty (-\lambda \overline{w})^k A^{[k]} g(z),\\
        \widetilde{V_{(g, \varphi)}}(z) &= e^{-|z|^2 + a \overline{z} + \lambda |z|^2} \sum_{k=0}^\infty (-\lambda \overline{z})^k A^{[k]} g(z),\\
        K_{V_{(g,\varphi)}}(w,z) &= e^{\varphi(z) \cdot \overline{w}} \sum_{k=0}^\infty (-\lambda \overline{w})^k A^{[k]} g(z).
    \end{align}
\end{corollary}

\subsection{Toeplitz-type operators}\label{sec:toeplitztype}
In \cite{Xu_Yu2023, Xu_Yu2024}, \emph{Toeplitz-type operators} (also called \emph{generalized Toeplitz operators}) were considered. We briefly recall their definition. For $z \in \mathbb C$, we let
\begin{align*}
    U_z f(w) = f(z-w) k_z(w).
\end{align*}
It is not hard to verify that $U_z$ is a bounded, self-adjoint operator on $F^2$. Given $f\in L^\infty(\mathbb C)$, the Toeplitz-type operator $T_f^{(j)}$ (where $j \in \mathbb N_0$) is defined weakly as
\begin{align*}
    \langle T_f^{(j)} g, h\rangle := \frac{1}{\pi} \int_{\mathbb C} f(z) \langle U_z g, e_j\rangle \langle e_j, U_z h\rangle ~dz.
\end{align*}
In \cite{Xu_Yu2023} it was proven that these operators are bounded for every $f \in L^\infty(\mathbb C)$ and compactness of these operators is characterized in terms of the behaviour of their Berezin transform. We want to revisit these results from the perspective of the phase space formalism of \emph{quantum harmonic analysis (QHA)}, which is a very convenient formalism to study Toeplitz operators on the Fock space \cite{Fulsche2020}.

For doing so, we first observe that
\begin{align*}
    U_z f = W_z R f, 
\end{align*}
where $R$ is the parity operator $Rf(z) = f(-z)$. Note that this operator satisfies $W_z R = R W_{-z}$ for all $z \in \mathbb C$. Then, for the term occurring in the definition of $T_f^{(j)}$ we have
\begin{align*}
    \langle U_z g e_j\rangle \langle e_j U_z h\rangle &= \langle (e_j \otimes e_j) W_z R g, W_z R h\rangle\\
    &= \langle W_z R (e_j \otimes e_j) R W_{-z} g, h\rangle.
\end{align*}
In the following, we will use the conventions for the phase space translation $\alpha_z$ and $\beta_-$ taken from \cite{Fulsche2020, werner84}: For $A \in \mathcal L(\mathcal H)$ and $z \in \mathbb C$ we define
\begin{align*}
    \alpha_z(A) := W_z A W_{-z}, \quad \beta_-(A) := RAR.
\end{align*}
We will also use the convolution between suitable operators and functions, which is defined as:
\begin{align*}
    f \ast A := \int_{\mathbb C} f(z) \alpha_z(A)~dz.
\end{align*}
This integral has to be understood in the weak sense. As a matter of fact, this convolution always results in a bounded operator, which is automatically contained in $\mathcal T$, whenever $f$ is bounded and $A$ is trace class (cf.\ \cite[Lemma 2.8, Theorem 3.1]{Fulsche2020}). 
With these notations, the definition of $T_f^{(j)}$ could have been written as:
\begin{align*}
    T_f^{(j)} :&= \frac{1}{\pi} \int_{\mathbb C} f(z) W_z R (e_j \otimes e_j) R W_{-z}~dz\\
    &= \frac{1}{\pi} \int_{\mathbb C} f(z) \alpha_z (\beta_-(e_j \otimes e_j))~dz\\
    &= \frac{1}{\pi} f \ast \beta_-(e_j \otimes e_j).
\end{align*}

Now, we observe that $Re_j = (-1)^{j} e_j$ such that $R(e_j \otimes e_j)R = e_j \otimes e_j$. We arrive at the following fact:
\begin{proposition}
    Let $f \in L^\infty(\mathbb C)$ and $j \in \mathbb N_0$. Then, the following holds true:
    \begin{align*}
        T_f^{(j)} = \frac{1}{\pi} f \ast (e_j \otimes e_j) \in \mathcal T.
    \end{align*}
\end{proposition}
As an immediate consequence, we obtain from Theorem \ref{thm:compactness}:
\begin{proposition}
    Let $f \in L^\infty(\mathbb C)$ and $j \in \mathbb N_0$. Then, $T_f^{(j)}$ is compact if and only if $\widetilde{T_f^{(j)}} \in C_0(\mathbb C)$. 
\end{proposition}
We now make use of the convolution of two bounded operators $A, B$ (one of which has to be trace class), which yields a function on $\mathbb C$:
\begin{align*}
    A \ast B(z) := \tr(A \alpha_z(\beta_-(B))).
\end{align*}
Then, all the notions of convolutions introduced above are associative \cite[Lemma 2.7]{Fulsche2020}. Further, the Berezin transform of an operator $A$ is given by $(e_0 \otimes e_0) \ast A$ \cite[Lemma 2.10]{Fulsche2020}. Therefore, we see that
\begin{align*}
    \widetilde{T_f^{(j)}} = \frac{1}{\pi} f \ast [(e_j \otimes e_j) \ast (e_0 \otimes e_0)].
\end{align*}
Computing
\begin{align*}
    (e_j \otimes e_j) \ast (e_0 \otimes e_0)(z) &= \tr((e_j \otimes e_j) W_z R (e_0 \otimes e_0) R W_{-z})\\
    &= \tr( (e_j \otimes e_j) (k_z \otimes k_z))\\
    &= e^{-|z|^2}|\langle e_j, K_z\rangle|^2\\
    &= e^{-|z|^2} \frac{|z|^{2j}}{j!}\\
    &=: g_j(z),
\end{align*}
we obtain that $\widetilde{T_f^{(j)}} = g_j \ast f$. By continuing this formula analytically, resp.\ anti-analytically, we obtain:
\begin{align*}
    \widetilde{T_f^{(j)}}(w,z) &= \frac{1}{\pi j!}\int_{\mathbb C} f(v) (z-v)\overline{(w-v)}e^{-(z-v)\overline{(w-v)}}~dw,\\
    K_{T_f^{(j)}}(w,z) &= e^{\frac{|z|^2 + |w|^2}{2}}\frac{1}{\pi j!}\int_{\mathbb C} f(v) (z-v)\overline{(w-v)}e^{-(z-v)\overline{(w-v)}}~dw.
\end{align*}
Combining the previous results with Theorem \ref{thm:compactness}, we conclude:
\begin{proposition}
    Let $f \in L^\infty(\mathbb C)$ and $j \in \mathbb N_0$. Then, $T_f^{(j)}$ is compact if and only if $f \ast g_j \in C_0(\mathbb C)$. 
\end{proposition}
Recall the notion of the \emph{Fourier-Weyl transform} of an operator $A$ (say, $A$ is trace class):
\begin{align*}
    \mathcal F_W(A)(\xi) := \tr(AW_{-\xi}), \quad \xi \in \mathbb C.
\end{align*}
Then, $\mathcal F_W$ is an injective map $\mathcal F_W: \mathcal T^1(F^2) \to C_0(\mathbb C)$ (where $\mathcal T^1(F^2)$ denotes the trace class on $F^2$). Recall that the Weyl operators $W_z$ satisfy the CCR relation
\begin{align}\label{eq:ccr}
    W_z W_w = e^{-i\operatorname{Im}(z\cdot \overline{w})}W_{z+w}, \quad z, w \in \mathbb C.
\end{align}
If $(\mathcal H, V_z)$ is any other irreducible projective representation of $\mathbb C$ such that
\begin{align*}
    V_z V_w = e^{-i\operatorname{Im}(z\cdot \overline{w})}V_{z+w}, \quad z, w \in \mathbb C,
\end{align*}
then there exists a unitary map $\mathfrak A: F^2 \to \mathcal H$ (which is unique up to a constant factor) such that $\mathfrak AW_z = V_z \mathfrak A$ for each $z \in \mathbb C$ (this is the theorem of Stone and von Neumann). If $A \in \mathcal T^1(F^2)$ and $B \in \mathcal T^1(\mathcal H)$ satisfy
\begin{align*}
    \mathcal F_W(A)(\xi) = \tr(AW_{-\xi}) = \tr(BV_{-\xi}) =: \mathcal F_W'(B)(\xi)
\end{align*}
for all $\xi \in \mathbb C$, then $A$ and $\mathfrak A^\ast B \mathfrak A$ have the same Fourier-Weyl transform. By injectivity of this map (\cite[Prop.\ 3.4]{werner84} or \cite[Prop.\ 5.17]{Fulsche_Galke2023}), we conclude that $A = \mathfrak A^\ast B \mathfrak A$. All this can be used as follows: First, we observe that
\begin{align*}
    \mathcal F_W(e_j \otimes e_j)(\xi) &= \tr((e_j \otimes e_j)W_{-\xi}) = \tr(e_j \otimes (W_\xi e_j))\\
    &= \langle e_j, W_{\xi} e_j\rangle\\
    &= \frac{1}{j!} \frac{1}{\pi} \int_{\mathbb C} z^j (\overline{z-\xi})^j e^{\xi \overline{z} - \frac{|\xi|^2}{2}} e^{-|z|^2}~dz\\
    &= e^{-\frac{|\xi|^2}{2}}\frac{1}{j!} \frac{1}{\pi} \sum_{k=0}^j \binom{j}{k} (-1)^k \overline{\xi}^k \int_{\mathbb C} z^j \overline{z}^{j-k} e^{\xi \overline{z}} e^{-|z|^2}~dz\\
    &= e^{-\frac{|\xi|^2}{2}}\frac{1}{j!} \frac{1}{\pi} \sum_{k=0}^j \binom{j}{k} (-1)^k \overline{\xi}^k \sum_{m=0}^\infty \int_{\mathbb C} z^j \overline{z}^{j-k}  \frac{(\xi \overline{z})^m}{m!} e^{-|z|^2}~dz\\
    &= e^{-\frac{|\xi|^2}{2}}\frac{1}{j!} \frac{1}{\pi} \sum_{k=0}^j \binom{j}{k} (-1)^k \overline{\xi}^k \sum_{m=0}^\infty  \frac{\xi^m}{m!} \int_{\mathbb C} z^j \overline{z}^{j-k+m}   e^{-|z|^2}~dz
\end{align*}
Recalling that
\begin{align*}
    \frac{1}{\pi} \int_{\mathbb C} z^a \overline{z}^b e^{-|z|^2}~dz = \begin{cases}
        0, \quad a \neq b,\\
        a!, \quad a = b,
    \end{cases}
\end{align*}
we obtain:
\begin{align*}
    \mathcal F_W(e_j \otimes e_j)(\xi) &= e^{-\frac{|\xi|^2}{2}} \sum_{k=0}^j \binom{j}{k} (-1)^k \overline{\xi}^k \frac{\xi^k}{k!}\\
    &= e^{-\frac{|\xi|^2}{2}} L_j^0(|\xi|^2),
\end{align*}
where $L_j^0$ is the $j$-th Laguerre polynomial
\begin{align*}
    L_k^0(x) = \sum_{k=0}^j \binom{j}{k} (-1)^k \frac{x^k}{k!}.
\end{align*}
Now, we briefly make a short detour to polyanalytic Fock spaces, without giving many details. For $j \in \mathbb N$ we denote by $F_j^2$ the subspace of all (smooth) functions $f$ in $L^2(\mathbb C, \mu)$ satisfying
\begin{align*}
    \frac{\partial^j}{\partial \overline{z}^j} f = 0.
\end{align*}
Then, $F_j^2$ is a closed subspace of $L^2(\mathbb C, \mu)$ and $F_1^2 = F^2$. Clearly, $F_j^2 \subset F_{j+1}^2$. We now decompose these spaces as follows:
\begin{align*}
    F_{j+1}^2 = F_j^2 \oplus F_{(j)}^2 = \oplus_{k=1}^{j+1} F_{(k)}^2,
\end{align*}
i.e., $F_{(j)}^2 := F_j^2 \ominus F_{j-1}^2$. The space $F^2_{(j)}, j \in \mathbb N$, is called the \emph{true $j$-th polyanalytic Fock space} on $\mathbb C$, cf.\ \cite{Fulsche_Hagger2023}. We recall that $F_{(j)}^2$ is also a closed subspace of $L^2(\mathbb C, \mu)$ on which the Weyl operators $W_z$ (defined as in Eq.\ \eqref{def:Weyloperator}) act irreducibly. They still satisfy the CCR relation \eqref{eq:ccr} over $F_{(j)}^2$. Hence, for each $j \in \mathbb N$ there exists a unitary operator $\mathfrak A_j: F^2 \to F_{(j)}^2$ such that $W_z \mathfrak A_j = \mathfrak A_j W_z$. By $k_{z, (j)}$, we refer to the normalized reproducing kernel of $F_{(j)}^2$.  Then, in \cite[Proposition 3.2(a)]{Fulsche_Hagger2023} it was shown that:
\begin{align*}
    \mathcal F_W'(k_{0, (j)} \otimes k_{0, (j)})(\xi) = e^{-\frac{|\xi|^2}{2}} L_{j-1}^0(|\xi|^2).
\end{align*}
By the above comment, we therefore conclude:
\begin{align*}
    e_{j-1} \otimes e_{j-1} = \mathfrak A_j^\ast (k_{0, (j)} \otimes k_{0, (j)}) \mathfrak A_j.
\end{align*}
In particular, since $\mathfrak A_j$ intertwines the Weyl operators, we see that:
\begin{align*}
    T_f^{(j)} &= \frac{1}{\pi} f \ast (e_j \otimes e_j) = \frac{1}{\pi} f \ast \mathfrak A_j^\ast (k_{0, (j)} \otimes k_{0, (j)}) \mathfrak A_j\\
    &= \frac{1}{\pi} \mathfrak A_j^\ast (f \ast (k_{0, (j)} \otimes k_{0, (j)})) \mathfrak A_j\\
    &= \mathfrak A_j^\ast T_{f, (j)} \mathfrak A_j,
\end{align*}
where $T_{f, (j)}$ denotes the Toeplitz operator with symbol $f$ on the $j$-th true polyanalytic Fock space (see, e.g., \cite[Eq.\ (3.7)]{Fulsche_Hagger2023} for the last identity). In particular, the Toeplitz-type operators $T_f^{(j)}$ are unitarily equivalent to the true polyanalytic Toeplitz operators $T_{f, (j)}$. Hence, all the properties discussed in \cite[Chapter 3]{Fulsche_Hagger2023} carry over verbatim to Toeplitz-type operators.

\subsection{Hausdorff operators}\label{sec:Hausdorff}
The investigation of Hausdorff operators has a long tradition, which we will not discuss here. The paper \cite{Galanopoulos_Stylogiannis2023} contains a nice historical summary on this topic, to which we refer. Instead, we will focus on discussing Hausdorff operators on $F^2$. 

Given a positive Borel measure $\rho$ on $(0, \infty)$, the Hausdorff operators $H_\rho$ formally acts as
\begin{align}\label{def:Hausdorff_operator}
    H_\rho f(z) = \int_{(0, \infty)} f\left( \frac zt \right) \frac 1t ~d\rho(t).
\end{align}
In \cite{Galanopoulos_Stylogiannis2023}, boundedness and compactness of $H_\rho$ on $F^2$ has been characterized. The main result is the following:
\begin{theorem}[{\cite[Theorems 1.1 and 1.2]{Galanopoulos_Stylogiannis2023}}]\label{thm:Hmu}
    $H_\rho$ acts boundedly on $F^2$ if and only if $\rho((0, 1)) = 0$ and $\int_{[1, \infty)} \frac{1}{t}~d\rho(t) < \infty$. In this case, its norm is given by 
\begin{align}\label{eq:Hausdorff_norm}
    \| H_\rho\| = \int_{[1, \infty)} \frac{1}{t}~d\rho(t).
\end{align}
    If $H_\rho$ is bounded, then it is compact if and only if $\rho(\{ 1\}) = 0$. 
\end{theorem}
We note that \cite{Galanopoulos_Stylogiannis2023} also proved that $F^2$ boundedness and compactness is equivalent to boundedness and compactness of $H_\rho$ on every $F^p$, $1 \leq p \leq \infty$ (see \cite{Bauer_Fulsche2020, Bauer_Isralowitz2012, Zhu2012} for definition of the $p$-Fock spaces). 

In \cite{Galanopoulos_Stylogiannis2023}, the properties of $H_\rho$ were discussed in terms of a certain multiplier problem. We will give a quick discussion on these operators from the perspective of \emph{localized operators}. Since we will make use of the norm equality \eqref{eq:Hausdorff_norm}, we first give a short proof of the boundedness characterization in the Hilbert space case for completeness. We refer here also to \cite{Galanopoulos_Stylogiannis2023}, where essentially the same proof is already presented.

\begin{proof}[Proof of the boundedness characterization in Theorem \ref{thm:Hmu}]
    We will first show that $H_\rho$ cannot be bounded whenever $\rho((0, 1)) > 0$. Under this assumption, since $\rho$ is a Borel measure, there exists a compact interval $[a, b] \subset (0, 1)$ such that $\rho([a, b]) > 0$. Let $\nu = \rho \cdot \mathbf 1_{[a, b]}$. By monotonicity, it suffices to prove that $H_\nu$ is unbounded. For doing so, we will show that the Berezin transform of $H_\nu$ is unbounded. 

    Note that the Berezin transform of $H_\nu$ is given by:
    \begin{align*}
        \widetilde{H_\nu}(z) &= \langle H_\nu k_z, k_z\rangle = e^{-|z|^2} \langle H_\nu K_z, K_z\rangle\\
        &= e^{-|z|^2} (H_\nu K_z)(z)\\
        &= e^{-|z|^2} \int_{[a,b]} e^{\frac{|z|^2}{t}} \frac{1}{t}~d\rho(t)\\
        &\geq e^{-|z|^2} \frac{1}{b} \int_{[a, b]}e^{\frac{|z|^2}{b}}~\frac{1}{t}~d\rho(t)\\
        &= e^{|z|^2(\frac{1}{b} - 1)} \frac{1}{b} \int_{[a,b]} \frac{1}{t}~d\rho(t).
    \end{align*}
    Since $\rho$ is a Borel measure on $(0, \infty)$, the last integral is finite, hence this expression is well-defined. Nevertheless, since $\frac{1}{b} > 1$, this function is unbounded. Hence, $H_\nu$ cannot be a bounded operator.

    We now assume that $\rho$ is such that $\rho((0, 1)) = 0$. Then, for $H_\rho$ being bounded, we necessarily need to have that $H_\rho 1(z)$ exists for every $z \in \mathbb C$. But this function is given by
    \begin{align}\label{Hmu_one}
        H_\rho 1(z) = \int_{[1, \infty)} \frac{1}{t} ~d\rho(t),
    \end{align}
    hence the integral needs to be bounded. 

    For sufficiency of the criterion, note that for the standard orthonormal basis $e_n(z) = \frac{z^n}{\sqrt{n!}}$, $n \in \mathbb N_0$, we have:
    \begin{align*}
        H_\rho e_n(z) &= \int_{[1, \infty)} \frac{z^n}{\sqrt{n!} t^n} \frac{1}{t}~d\rho(t)\\
        &= e_n(z) \int_{[1, \infty)} \frac{1}{t^{n+1}}~d\rho(t)\\
        &\leq e_n(z) \int_{[1, \infty)} \frac{1}{t}~d\rho(t).
    \end{align*}
    Applying this estimate to the Taylor series of an arbitrary $f \in F^2$ shows that $H_\rho$ is bounded with $\| H_\rho\| \leq \int_{[1, \infty)} \frac{1}{t}~d\rho(t)$. Comparing with \eqref{Hmu_one} shows equality of the norm and the integral.
\end{proof}
Given the characterization of boundedness for $H_\rho$, we will now improve that result: Whenever $H_\rho$ is bounded, it is automatically contained in $\mathcal T$. To prove this, we will make use of Proposition \ref{crit:Wiener_algebra}.
\begin{lemma}
    Assume that $\rho$ is such that $\rho((0, 1)) = 0$ and $\operatorname{supp}(\rho)$ is compact. Then, there exists $H \in L^1(\mathbb C)$ such that 
    \begin{align*}
        |\langle H_\rho k_z, k_w\rangle| \leq H(z-w).
    \end{align*}
    In particular, $H_\rho \in \mathcal T$.
\end{lemma}
\begin{proof}
    In the following proof we set $M_\rho := \int_{[1, \infty)} \frac{1}{t}~d\rho(t)$. For the bivariate Berezin transform, we have:
    \begin{align*}
        |\langle H_\rho k_z, k_w\rangle| &= e^{-\frac{|z|^2 + |w|^2}{2}} |H_\rho K_z(w)|\\
        &\leq e^{-\frac{|z|^2 + |w|^2}{2}} \int_{[1, \infty)} e^{\operatorname{Re}(\frac{w \overline{z}}{t})} \frac{1}{t}~d\rho(t).
    \end{align*}
    When $\operatorname{Re}(w \overline{z}) \geq 0$, then we clearly have:
    \begin{align*}
        |\langle H_\rho k_z, k_w\rangle| \leq e^{-\frac{|z|^2 + |w|^2}{2}} \int_{[1, \infty)} e^{\operatorname{Re}(w \overline{z})} \frac{1}{t}~d\rho(t) = e^{-\frac{|z-w|^2}{2}} M_\rho.
    \end{align*}
    For those $z, w$ such that $\operatorname{Re}(w \overline{z}) < 0$, we continue as follows: Assume that $\rho$ is supported in $[1, c]$ for some $c > 1$. Then,
    \begin{align*}
        |\langle H_\rho k_z, k_w\rangle| &\leq e^{-\frac{|z|^2 + |w|^2}{2}} \int_{[1, c]} e^{\operatorname{Re}(\frac{w \overline{z}}{t})} \frac{1}{t}~d\rho(t)\\
        &\leq e^{-\frac{|z|^2 + |w|^2}{2}} \int_{[1, c]} e^{\operatorname{Re}(\frac{w \overline{z}}{c})} \frac{1}{t}~d\rho(t)\\
        &= e^{\frac{|z|^2 + |w|^2}{2}(\frac{1}{c} - 1)} e^{-\frac{|z-w|^2}{2c}} M_\rho
    \end{align*}
    Since $\frac 1c - 1 < 0$, the first exponential is bounded by $1$. Hence, we see that for $\rho$ supported in $[1, c]$ we have that:
    \begin{align*}
        |\langle H_\rho k_z, k_w\rangle| \leq e^{-\frac{|z-w|^2}{2c}} M_\rho
    \end{align*}
    for all $z, w \in \mathbb C$. By Proposition \ref{crit:Wiener_algebra}, it follows that $H_\rho \in \mathcal T$.
\end{proof}
\begin{proposition}\label{prop:bounded_HO_T_algebra}
    Let $\rho$ be a positive Borel measure on $(0, \infty)$. Then, the following are equivalent:
    \begin{enumerate}
        \item $\rho((0, 1)) = 0$ and $\int_{[1, \infty)} \frac{1}{t}~d\rho(t) < \infty$;
        \item $H_\rho$ is bounded on $F^2$;
        \item $H_\rho \in \mathcal T$.
    \end{enumerate}
\end{proposition}
\begin{proof}
    Of course, we only need to prove the implication (2) $\Rightarrow$ (3). By the previous result, we know that $H_\rho \in \mathcal T$ whenever $\rho$ is compactly supported in $[1, \infty)$. Hence, we know that for $\rho_n = \rho \mathbf 1_{[1, n)}$ we have that $H_{\rho_n} \in \mathcal T$ for every $n \in \mathbb N$. Further, 
    \begin{align*}
        \| H_\rho - H_{\rho_n}\| = \| H_{\rho \mathbf 1_{[n, \infty)}}\| = \int_{[n, \infty)} \frac{1}{t}~d\rho(t).
    \end{align*}
    Since the integral $\int_{[1, \infty)} \frac{1}{t}~d\rho(t)$ is finite, we know that $\int_{[n, \infty)} \frac{1}{t}~d\rho(t) \to 0$ as $n \to \infty$. Hence, $H_\rho \in \mathcal T$. 
\end{proof}
Knowing that $H_\rho$ is always contained in $\mathcal T$ whenever it is bounded, it is now an easy task to give a new proof for the compactness result, using Theorem \ref{thm:compactness}:

\begin{proof}[Proof of the compactness characterization in Theorem \ref{thm:Hmu}]
Indeed, we only have to calculate the limit $\lim_{z \to \infty} \widetilde{H_\rho}(z)$. Writing
\begin{align*}
    \widetilde{H_\rho}(z) = \rho(\{1\}) \cdot 1 + \int_{(1, \infty)} e^{\left( \frac{1}{t} - 1 \right)|z|^2}\frac{1}{t}~d\rho(t), 
\end{align*}
we see that (using the dominated convergence theorem with dominating function $\frac{1}{t}$) 
\begin{align*}
    \lim_{|z| \to \infty} \int_{(1, \infty)} e^{\left( \frac{1}{t} - 1 \right)|z|^2}\frac{1}{t}~d\rho(t) = 0,
\end{align*}
showing that
\begin{align*}
    \lim_{z \to \infty} \widetilde{H_\rho}(z) = \rho(\{ 1\}).
\end{align*}
This yields the compactness characterization for these operators. 
\end{proof}

Defining for $\theta \in S^1$ the operator $D_\theta$ as $D_\theta f(z) = f(\theta z)$, one finds that $D_\theta H_\rho = H_\rho D_\theta$, i.e., $H_\rho$ is a radial operator. The spectral decomposition of every radial operator on the Fock space is given by the eigenvalue sequence with respect to the standard orthonormal basis $e_n(z) = z^n/\sqrt{n!}$. In our case:
\begin{align*}
    a_n &= \frac{1}{(n!)^2}\langle H_\rho z^n, z^n\rangle = \frac{1}{n!} \frac{1}{\pi}\int_{\mathbb C} \int_{[1, \infty)} \frac{z^n}{t^n} \frac{1}{t}~d\rho(t)~\overline{z}^n e^{-|z|^2}~dz\\
    &= \frac{1}{n!} \frac{1}{\pi} \int_{\mathbb C} |z|^{2n} e^{-|z|^2}~dz \int_{[1, \infty)} \frac{1}{t^{n+1}}~d\rho(t)\\
    &= \int_{[1, \infty)} \frac{1}{t^{n+1}}~d\rho(t).
\end{align*}
It is not hard to verify that this sequence converges (towards $\rho(\{ 1\})$, which can be utilized to prove the compactness characterization as done in \cite{Galanopoulos_Stylogiannis2023}). Therefore, it is slowly oscillating with respect to the square root metric
\begin{align*}
    d(m,n) = |\sqrt{m} - \sqrt{n}|, \quad m, n \in \mathbb N_0. 
\end{align*}
Hence, the main result of \cite{Esmeral_Maximenko2016} can also be utilized to show that $H_\rho \in \mathcal T$. 

We now want to discuss the following question: When is the operator $H_\rho$ a Toeplitz operator, i.e., when does there exist a (necessarily radial) $f \in L^\infty(\mathbb C)$ such that $H_\rho = T_f$? To this end, we recall that for a Toeplitz operator with radial symbol $f$ (i.e., $f(z) = g(|z|)$ for some $g \in L^\infty([0, \infty))$) its eigenvalue sequence is given by
\begin{align}\label{egv_sequence_radial_Toeplitz}
    a_n = \frac{1}{n!} \int_0^\infty g(\sqrt{s}) s^n e^{-s}~ds,
\end{align}
cf.\ \cite{Grudsky_Vasilevski2002}. Hence, given any $\rho$ satisfying the conditions for boundedness of $H_\rho$, the task consists in finding $g \in L^\infty([0, \infty))$ such that for all $n \in \mathbb N_0$:
\begin{align*}
    \int_{[1, \infty)} \frac{1}{t^{n+1}}~d\rho(t) = \frac{1}{n!} \int_0^\infty g(\sqrt{s})s^ne^{-s}~ds.
\end{align*}
To the best of our knowledge, a general characterization of those sequences $(a_n) \in \ell^\infty(\mathbb N_0)$ solving the version of the Stieltjes moment problem \eqref{egv_sequence_radial_Toeplitz} is not yet available.

First note that for $\rho = \delta_1$, the operator is $H_\rho$ is simply the identity, i.e., $H_{\delta_1} = T_1$. Hence, the problem of solving $H_\rho = T_f$ for $f$ can be reduced to the case of compact operators $H_\rho$. Further, by the identities
\begin{align*}
    \int_{[1, \infty)} \frac{1}{t^{n+1}}~d\delta_x(t) = \frac{1}{x^{n+1}} =  \int_0^\infty s^{n}e^{-xs}~dx = \int_0^\infty e^{-(x-1)s} s^n e^{-s}~ds,
\end{align*}
we see that for every $x \geq 1$ it is $H_{\delta_x} = T_{f_x}$ with $f_x(z) = e^{-(x-1)|z|^2}$. This solves the question entirely for discrete measures. We are left with the following:
\begin{question}
Is there a characterization of those continuous measures $\rho$ such that $H_\rho$ is a Toeplitz operator?
\end{question}

We want to end this discussion on Hausdorff operators by noting that it is indeed not necessary to consider only real measures $\rho$ in the definition of the Hausdorff operator. Indeed, for example making use of the Hahn-Jordan decomposition for a complex measure $\rho$ and denoting the the total variation measure of $\rho$ by $|\rho|$, one immediately obtains:
\begin{proposition}
    Let $\rho$ be a complex measure on $[1, \infty)$ such that $\int_{[1, \infty)} \frac{1}{t}~d|\rho|(t)< \infty$, then the operator $H_\rho$ defined by \eqref{def:Hausdorff_operator} is bounded with $\| H_\rho\| \leq \int_{[1, \infty)}\frac{1}{t}~d|\rho|(t)$. In that case, it holds true that $H_\rho \in \mathcal T$. $H_\rho$ is compact if and only if $|\rho|(\{ 1\}) = 0$. 
\end{proposition}

\bibliographystyle{abbrv}
\bibliography{main}

\begin{thebibliography}{10}

\bibitem{Bais_Venku2023}
S.~R. Bais and N.~D. Venku.
\newblock $\mathcal{L}$-invariant and radial singular integral operators on the
  {F}ock space.
\newblock {\em J. Pseudo-Differ. Oper. Appl.}, 14:11, 2023.

\bibitem{Bauer_Fulsche2020}
W.~Bauer and R.~Fulsche.
\newblock {Berger-Coburn theorem, localized operators, and the Toeplitz
  algebra}.
\newblock In W.~Bauer, R.~Duduchava, S.~Grudsky, and M.~A. Kaashoek, editors,
  {\em Operator Algebras, Toeplitz Operators and Related Topics}, pages 53--77.
  Springer International Publishing, Cham, 2020.

\bibitem{Bauer_Isralowitz2012}
W.~Bauer and J.~Isralowitz.
\newblock {C}ompactness characterization of operators in the {T}oeplitz algebra
  of the {F}ock space {$F_\alpha^p$}.
\newblock {\em J. Funct. Anal.}, 263:1323--1355, 2012.

\bibitem{Berezin1972}
F.~A. Berezin.
\newblock {Covariant and contravariant symbols of operators}.
\newblock {\em Math. USSR Izv.}, 6:1117--1151, 1972.

\bibitem{Berezin1974}
F.~A. Berezin.
\newblock Quantization.
\newblock {\em Izv. Akad. Nauk SSSR Ser. Mat.}, 38:1116–1175, 1974.

\bibitem{Berger_Coburn1987}
C.~L. Berger and L.~A. Coburn.
\newblock {Toeplitz operators on the Segal-Bargmann space}.
\newblock {\em Trans. Amer. Math. Soc.}, 301:813--829, 1987.

\bibitem{Berger_Coburn1994}
C.~L. Berger and L.~A. Coburn.
\newblock {Heat flow and Berezin-Toeplitz estimates}.
\newblock {\em Amer. J. Math.}, 116:563--590, 1994.

\bibitem{Bommier-Hato_Youssfi_Zhu2015}
H.~Bommier-Hato, E.~H. Youssfi, and K.~Zhu.
\newblock {Sarason's Toeplitz product problem for a class of Fock spaces}.
\newblock {\em Bull. Sci. Math.}, 141:408–442, 2017.

\bibitem{Borichev_Rozenblum2015}
A.~Borichev and G.~Rozenblum.
\newblock {The finite rank theorem for Toeplitz operators on the Fock space}.
\newblock {\em J. Geom. Anal.}, 25:347–356, 2015.

\bibitem{Cao_Li2020}
G.~Cao, J.~Li, M.~Shen, B.~D. Wick, and L.~Yan.
\newblock {A boundedness criterion for singular integral operators of
  convolution type on the Fock space}.
\newblock {\em Adv. Math.}, 363:107001, 2020.

\bibitem{Carroll_Gilmore2021}
T.~Carroll and C.~Gilmore.
\newblock {Weighted composition operators on Fock spaces and their dynamics}.
\newblock {\em J. Math. Anal. Appl.}, 502:125234, 2021.

\bibitem{Chen_Wang2022}
W.~Chen and E.~Wang.
\newblock {Equivalent norms on generalized Fock spaces and the extended
  Ces\`{a}ro operators}.
\newblock {\em Complex Anal. Oper. Theory}, 16:62, 2022.

\bibitem{Cho_Park_Zhu2014}
H.~R. Cho, J.-D. Park, and K.~Zhu.
\newblock {Products of Toeplitz operators on the Fock space}.
\newblock {\em {Proc. Amer. Math. Soc.}}, 142:2483--2489, 2014.

\bibitem{Choe_Izuchi_Koo2010}
B.~R. Choe, K.~H. Izuchi, and H.~Koo.
\newblock {Linear sums of two composition operators on the Fock space}.
\newblock {\em {J. Math. Anal. Appl.}}, 369:112–119, 2010.

\bibitem{Coburn2001}
L.~A. Coburn.
\newblock {On the Berezin-Toeplitz calculus}.
\newblock {\em Proc. Amer. Math. Soc.}, 129:3331–3338, 2001.

\bibitem{Constantin_2012}
O.~Constantin.
\newblock {A Volterra-type integration operator on Fock spaces}.
\newblock {\em Proc. Amer. Math. Soc.}, 140(12):4247--4257, 2012.

\bibitem{Dong_Zhu2022}
X.~Dong and K.~Zhu.
\newblock {Canonical integral operators on the Fock space}.
\newblock {\em Math. Z.}, 306:64, 2024.

\bibitem{Esmeral_Maximenko2016}
K.~Esmeral and E.~A. Maximenko.
\newblock Radial {T}oeplitz operators on the {F}ock space and
  square-root-slowly oscillating sequences.
\newblock {\em Complex Anal. Oper. Theory}, 10:1655–1677, 2016.

\bibitem{Esmeral_Vasilevski2016}
K.~Esmeral and N.~L. Vasilevski.
\newblock ${C}^*$-algebra {g}enerated by {h}orizontal {T}oeplitz {o}perators on
  the {F}ock {s}pace.
\newblock {\em {Bol. Soc. Mat. Mex.}}, 22:567–582, 2016.

\bibitem{Fekadiea2023}
Z.~A. Fekadiea, T.~Mengestie, and M.~H. Takele.
\newblock {Closed range integral operators on Fock spaces}.
\newblock {\em Complex Anal. Oper. Theory}, 17:107, 2023.

\bibitem{Fleming_Jamison_2002}
R.~J. Fleming and J.~E. Jamison.
\newblock {\em {Isometries on Banach spaces: function spaces}}.
\newblock {Chapman and Hall/CRC}, 2002.

\bibitem{Fulsche2023}
R.~Fulsche.
\newblock {A Wiener algebra for Fock space operators}.
\newblock preprint available on arXiv:2311.11859.

\bibitem{Fulsche2020}
R.~Fulsche.
\newblock {Correspondence theory on $p$-Fock spaces with applications to
  Toeplitz algebras}.
\newblock {\em J. Funct. Anal.}, 279:108661, 2020.

\bibitem{Fulsche2024}
R.~Fulsche.
\newblock Toeplitz operators on non-reflexive {F}ock spaces.
\newblock {\em Rev. Mat. Iberoam.}, 40:1115–1148, 2024.

\bibitem{Fulsche_Galke2023}
R.~Fulsche and N.~Galke.
\newblock Quantum harmonic analysis on locally compact abelian groups.
\newblock preprint available on arXiv:2308.02078, 2023.

\bibitem{Fulsche_Hagger2023}
R.~Fulsche and R.~Hagger.
\newblock {Quantum harmonic analysis for polyanalytic Fock spaces}.
\newblock unpublished preprint available at arXiv:2308.11292.

\bibitem{Fulsche_Rodriguez_2023}
R.~Fulsche and M.~A. Rodriguez~Rodriguez.
\newblock {Commutative $G$-invariant Toeplitz C$^\ast$-algebras on the Fock
  space and their Gelfand theory through quantum harmonic analysis}.
\newblock Accepted for publication in J. Operator Theory. Preprint availabe at
  arXiv:2307.15632, 2023.

\bibitem{Galanopoulos_Stylogiannis2023}
P.~Galanopoulos and G.~Stylogiannis.
\newblock {Hausdorff operators on Fock spaces and a coefficient multiplier
  problem}.
\newblock {\em Proc. Amer. Math. Soc.}, 151:3023–3035, 2023.

\bibitem{Grudsky_Vasilevski2002}
S.~M. Grudsky and N.~L. Vasilevski.
\newblock {Toeplitz operators on the Fock space: Radial component effects}.
\newblock {\em Integr. Equ. Oper. Theory}, 40:10–37, 2002.

\bibitem{Hagger2021}
R.~Hagger.
\newblock {Essential commutants and characterizations of the Toeplitz algebra}.
\newblock {\em J. Operator Theory}, 86:125–143, 2021.

\bibitem{Hu_Lv2014}
Z.~Hu and X.~Lv.
\newblock {Toeplitz operators on Fock spaces $F^p(\varphi)$}.
\newblock {\em Integr. Equ. Oper. Theory}, 80:33–59, 2014.

\bibitem{Hu_Lv2017}
Z.~Hu and X.~Lv.
\newblock {Positive Toeplitz operators between different doubling Fock spaces}.
\newblock {\em Taiwanese J. Math.}, 21:467–487, 2017.

\bibitem{Isralowitz_Virtanen_Wolf2015}
J.~Isralowitz, J.~Virtanen, and L.~Wolf.
\newblock {Schatten class Toeplitz operators on generalized Fock spaces}.
\newblock {\em J. Math. Anal. Appl.}, 421:329–337, 2015.

\bibitem{Janson_Peetre_Rochberg1987}
S.~Janson, J.~Peetre, and R.~Rochberg.
\newblock {H}ankel {f}orms and the {F}ock {s}pace.
\newblock {\em Rev. Mat. Iberoam.}, 3(1):61--138, 1987.

\bibitem{Jin_Tang_Feng2022}
J.~Jin, S.~Tang, and X.~Feng.
\newblock Hilbert-type operators acting between weighted {F}ock spaces.
\newblock {\em Complex Anal. Oper. Theory}, 16:103, 2022.

\bibitem{Le2014}
T.~Le.
\newblock {Normal and isometric weighted composition operators on the Fock
  space}.
\newblock {\em Bull. London Math. Soc.}, 46:847–856, 2014.

\bibitem{Mengestie_Worku2018}
T.~Mengestie and M.~Worku.
\newblock {Topological structures of generalized Volterra-type integral
  operators}.
\newblock {\em Mediterr. J. Math.}, 15:42, 2018.

\bibitem{Quiroga_Vasilevski2007}
R.~Quiroga-Barranco and N.~Vasilevski.
\newblock {C}ommutative ${C}^*$-{a}lgebras of {T}oeplitz {o}perators on the
  {u}nit {b}all, {I}. {B}argmann-{t}ype {t}ransforms and {s}pectral
  {r}epresentations of {T}oeplitz {o}perators.
\newblock {\em Integr. Equ. Oper. Theory}, 59:379--419, 2007.

\bibitem{Siskakis_2004}
A.~Siskakis.
\newblock {Volterra operators on spaces of analytic functions—a survey}.
\newblock In A.~Montes-Rodr\'{i}guez, editor, {\em {Proceedings of the First
  Advanced Course in Operator Theory and Complex Analysis}}, page 51–68.
  Universidad de Sevilla. Secretariado de Publicaciones, Seville, 2006.

\bibitem{Tien_Khoi2019}
P.~T. Tien and L.~H. Khoi.
\newblock {Differences of weighted composition operators between the Fock
  spaces}.
\newblock {\em Monatsh. Math.}, 188:183–193, 2019.

\bibitem{Ueki2007}
S.-I. Ueki.
\newblock {Weighted composition operator on the Fock space}.
\newblock {\em Proc. Amer. Math. Soc.}, 135:1405–1410, 2007.

\bibitem{werner84}
R.~Werner.
\newblock {Quantum harmonic analysis on phase space}.
\newblock {\em J. Math. Phys.}, 25:1404–1411, 1984.

\bibitem{Xia2015}
J.~Xia.
\newblock {Localization and the Toeplitz algebra on the Bergman space}.
\newblock {\em J. Funct. Anal.}, 269:781–814, 2015.

\bibitem{Xu_Yu2023}
C.~Xu and T.~Yu.
\newblock {Toeplitz-type operators on the {F}ock space $F_\alpha^2$}.
\newblock {\em Bull. Korean Math. Soc.}, 60:957–969, 2023.

\bibitem{Xu_Yu2024}
C.~Xu and T.~Yu.
\newblock {The generalized Toeplitz operators on the Fock space $F_\alpha^2$}.
\newblock {\em Czechoslovak Math. J.}, 74:231–246, 2024.

\bibitem{Zhu2012}
K.~Zhu.
\newblock {\em Analysis on {F}ock {S}paces}, volume 263 of {\em Graduate Texts
  in Mathematics}.
\newblock Springer US, New York, 2012.

\bibitem{Zhu2015}
K.~Zhu.
\newblock {Singular integral operators on the Fock space}.
\newblock {\em Integr. Equ. Oper. Theory}, 81:451–454, 2015.

\bibitem{Zhu2019}
K.~Zhu.
\newblock {Towards a dictionary for the Bargmann transform}.
\newblock In K.~Zhu, editor, {\em Handbook of analytic operator theory}, page
  319–349. CRC Press, 2019.

\end{thebibliography}

\vspace{0.5cm}
\begin{multicols}{2}
\noindent Wolfram Bauer\\
\href{bauer@math.uni-hannover.de}{\Letter ~bauer@math.uni-hannover.de}\\
\noindent Robert Fulsche\\
\href{fulsche@math.uni-hannover.de}{\Letter ~fulsche@math.uni-hannover.de}\\
\noindent Miguel Angel Rodriguez Rodriguez\\
\href{rodriguez@math.uni-hannover.de}{\Letter ~rodriguez@math.uni-hannover.de}
\\
All authors:\\
Institut f\"{u}r Analysis\\
Leibniz Universit\"at Hannover\\
Welfengarten 1\\
30167 Hannover\\
GERMANY

\end{multicols}

\end{document}